\documentclass[12pt,reqno]{amsart}

\usepackage{hyperref}

\usepackage[usenames,dvipsnames,svgnames,table]{xcolor}
\usepackage[all]{xy}

\usepackage{tikz}
\usepackage{tikz-cd}

\usepackage{amssymb}
\usepackage{mathrsfs}

\newtheorem{teo}{Theorem}[section]
\newtheorem{prop}[teo]{Proposition}

\newtheorem{lemma}[teo]{Lemma}

\theoremstyle{definition}
\newtheorem{defin}[teo]{Definition}
\newtheorem{remark}[teo]{Remark}

\numberwithin{equation}{section}

\newcommand{\C}{\mathbb{C}}
\newcommand{\R}{\mathbb{R}}

\newcommand{\alfa}{\alpha}
\newcommand{\alf}{\alpha}
\newcommand{\restr}[1]{\vert_{#1}}

\newcommand{\om}{\omega}
\newcommand{\cd}{\cdot}

\newcommand{\OO}{\mathcal{O}}

\newcommand{\debar}{\overline{\partial } }

\newcommand{\ag}{\mathsf{A}_g}
\newcommand{\mg}{\mathsf{M}_g}

\newcommand{\pet}{\eta}
\newcommand{\teta}{\widehat{\eta}}
\newcommand{\esse}{{\mathcal S}}
\newcommand{\Ci}{{\mathcal C}}
\newcommand{\nc}{\newcommand} 
\nc{\cA}{{\mathcal A}}
\nc{\cC}{{\mathcal C}}
\nc{\cO}{{\mathcal O}}
\nc{\cS}{{\mathcal S}}
\nc{\cN}{{\mathcal N}}
\nc{\bC}{{\mathbb C}}
\nc{\bZ}{{\mathbb Z}}

\begin{document}

\title[A Hodge theoretic projective structure]{A Hodge theoretic projective structure on
compact Riemann surfaces}

\author[I. Biswas]{Indranil Biswas}

\address{School of Mathematics, Tata Institute of Fundamental
Research, Homi Bhabha Road, Mumbai 400005, India}

\email{indranil@math.tifr.res.in}

\author[E. Colombo]{Elisabetta Colombo}

\address{Dipartimento di Matematica, Universit\`a di Milano, via Saldini 50,
I-20133, Milano, Italy}

\email{elisabetta.colombo@unimi.it}

\author[P. Frediani]{Paola Frediani}

\address{Dipartimento di Matematica, Universit\`a di Pavia,
via Ferrata 5, I-27100 Pavia, Italy}

\email{paola.frediani@unipv.it}

\author[G. P. Pirola]{Gian Pietro Pirola}

\address{Dipartimento di Matematica, Universit\`a di Pavia,
via Ferrata 5, I-27100 Pavia, Italy}

\email{gianpietro.pirola@unipv.it}



\begin{abstract}
Given any compact Riemann surface $C$, there is a canonical meromorphic 2--form $\widehat\eta$
on $C\times C$, with pole of order two on the diagonal $\Delta\, \subset\, C\times C$, constructed in \cite{cfg}. 
This meromorphic 2--form $\widehat\eta$ produces a canonical projective structure on $C$. 
On the other hand the uniformization theorem provides another canonical projective structure on any compact Riemann surface $C$. 
We prove that these two projective structures differ in general. This is done
by comparing the $(0,1)$--component of the differential of the corresponding sections of the moduli
space of projective structures over the moduli space of curves. The $(0,1)$--component of the
differential of the section corresponding to the projective
structure given by the uniformization theorem was computed by Zograf and Takhtadzhyan in \cite{ZT} as the Weil--Petersson
K\"ahler form $\omega_{wp}$ on the moduli space of curves. 
We prove that the $(0,1)$--component of the differential of the section of the moduli space of projective structures
corresponding to $\widehat{\eta}$ is the pullback of a nonzero constant scalar
multiple of the Siegel form, on the moduli space
of principally polarized abelian varieties, by the Torelli map.

\medskip
\noindent
\textsc{R\'esum\'e.}\, Pour toute surface de Riemann compacte $C$ donn\'ee, il existe une 2--forme 
m\'eromorphe canonique $\widehat\eta$ sur le produit $C \times C$, avec des p\^oles d'ordre deux sur la diagonale 
$\Delta\, \subset\, C\times C$, construite dans \cite{cfg}.  Cette 2--forme m\'eromorphe $\widehat\eta$ munit $C$ 
d'une structure projective canonique. Par ailleurs, le th\'eor\`eme d'uniformisation produit une autre structure 
projective canonique sur toute surface de Riemann compacte $C$. Nous d\'emontrons qu'en g\'en\'eral ces deux 
structures projectives sont distinctes. Ceci est prouv\'e en consid\'erant les diff\'erentielles des deux sections 
correspondantes de l'espace des modules de structures projectives au-dessus de l'espace des modules de courbes 
complexes et en comparant leur composantes de type $(0,1)$. La composante de type $(0,1)$ de la diff\'erentielle de 
la section correspondant au th\'eor\`eme d'uniformisation a \'et\'e identifi\'ee par Zograf et Takhtadzhyan dans 
\cite{ZT}, comme \'etant la forme de K\"ahler de Weil--Petersson $\omega_{wp}$ sur l'espace des modules des courbes. 
Nous prouvons que la composante de type $(0,1)$ de la diff\'erentielle de la section de l'espace des modules de 
structures projectives correspondant \`a $\widehat{\eta}$ est l'image r\'eciproque par l'application de Torelli d'un 
multiple scalaire non nul de la forme de Siegel sur l'espace des modules de vari\'et\'es ab\'eliennes munies d'une 
polarisation principale.
\end{abstract}

\maketitle

\tableofcontents

\section{Introduction}\label{se1}

Let $C$ be a compact connected Riemann surface, and let $\Delta\,\subset\, C\times C\, =:\, S$ be the
(reduced) diagonal divisor. There is a canonical meromorphic 2-form on $S$
\begin{equation}\label{i0}
\widehat{\eta} \,\in \,H^0(S, \, \Omega^2_S\otimes {\mathcal O}_S(2\Delta))
\end{equation}
constructed in \cite{cfg}. This form governs the second fundamental form of the Torelli map (\cite{cfg}) and its
description is recalled in Section \ref{se2.1}. 
The form
$\widehat{\eta}$ has the following two properties \cite{cfg}:
\begin{enumerate}
\item The restriction of $\widehat{\eta}$ to $\Delta$ coincides with the section of
$(\Omega^2_S\otimes {\mathcal O}_S(2\Delta))\vert_\Delta\,=\, {\mathcal O}_\Delta$ given by the constant
function $1$ on $\Delta$.

\item The involution of $C\times C$ defined by $(x_1,\, x_2)\, \longmapsto\, (x_2,\, x_1)$ takes
$\widehat{\eta}$ to $-\widehat{\eta}$.
\end{enumerate}
It is known that any holomorphic section of $\Omega^2_S\otimes {\mathcal O}_S(2\Delta)$ satisfying the above
two conditions produces a projective structure on $C$ (a summary on projective structures is
given in Section \ref{se5}). Therefore,
$\widehat{\eta}$ in \eqref{i0} produces a projective structure on $C$ which is canonical in the sense that
it does not depend on the choice of any structure or object on $C$.

There are many standard projective structures on a compact Riemann surface: Bers uniformization, Schottky 
uniformization, Earle uniformization, the standard Poincar\'e--Koebe uniformization etc. Among these only the last 
one, given by the standard uniformization (Poincar\'e--Koebe), is canonical in the above sense. Therefore, a natural 
question to ask is whether the projective structure given by $\widehat{\eta}$ in \eqref{i0} coincides with the 
projective structure given by the uniformization theorem. While the general guess has been that they do coincide, our 
first main result negates it. To explain it, denote by $\mg$ the moduli space of curves of genus $g$, and by 
${\mathcal P}_g$ be the moduli space parametrizing the isomorphism classes of pairs given by a genus $g$ curve and a 
projective structure on it. The forgetful map ${\mathcal P}_g\, \longrightarrow\, \mg$
that forgets the projective structure is a holomorphic fiber bundle 
(with respect to the orbifold structure on $\mg$). We prove the following (see Theorem \ref{lem-ch} and Remark 
\ref{rgenus2} for the case of $g\,=\,2$):

\begin{teo}\label{th1}
Assume that $g\, \geq\, 2$, and let $\beta^\eta$ and $\beta^u$ be the $C^\infty$ sections of
the holomorphic fiber bundle ${\mathcal P}_g\, \longrightarrow\, \mg$ given by $\widehat{\eta}$ in \eqref{i0}
and the uniformization theorem respectively. Then the two sections $\beta^\eta$ and $\beta^u$ do not
coincide.
\end{teo}

We now briefly describe the strategy of the proof of Theorem \ref{th1}.

The holomorphic fiber bundle ${\mathcal P}_g\, \longrightarrow\, \mg$ is a holomorphic torsor for
$\Omega^1_{\mg}$. Therefore, the differential of any $C^\infty$ section $\theta\, :\,
\mg\, \longrightarrow\, {\mathcal P}_g$ of it produces a $(1 \ ,1)$--form
$$
\overline{\partial}(\theta)\, \in\, C^\infty(\mg;\, \Omega^{1,1}_{\mg})\, .
$$
In fact, choosing a holomorphic trivialization of the torsor over an open subset $U\, \subset\, \mg$, the
restriction $\theta\vert_U$ gives a $(1,\, 0)$--form $\theta'$ on $U$, and
$(\overline{\partial}(\theta))\vert_U\,=\, \overline{\partial}\theta'$.
For the section $\beta^u$ in Theorem \ref{th1}, a theorem of Zograf and Takhtadzhyan says that
the $(1,\, 1)$--form $\overline{\partial}(\beta^u)$ on $\mg$ coincides with the Weil--Petersson
K\"ahler form $\omega_{wp}$ \cite{ZT}.

We carry out a thorough study of the $(1,\, 1)$--form $\overline{\partial}(\beta^\eta)$ on $\mg$
associated to the section $\beta^\eta$ in Theorem \ref{th1}. We prove that for any $g\, \geq\, 3$,
the form $\overline{\partial}(\beta^\eta)$ fails to be nondegenerate on the hyperelliptic locus of
$\mg$. In particular, $\overline{\partial}(\beta^\eta)$ does not coincide with the K\"ahler form
$\omega_{wp}$. This gives Theorem \ref{th1} for $g\, \geq\, 3$.

We obtain a more precise description 
of $\overline{\partial}(\beta^\eta)$ which is explained now.

Let $\mathsf{A}_g$ be the moduli space of principally polarized
complex abelian varieties of dimension $g$. Let
\begin{equation}\label{itau}
\tau\,:\, \mathsf{M}_g \,\longrightarrow\, \mathsf{A}_g
\end{equation}
be the Torelli map that sends a curve to its Jacobian equipped with the principal polarization given by a
theta line bundle. This map $\tau$ is an orbifold immersion outside the hyperelliptic locus \cite{os}.
The variety $\mathsf{A}_g$ is equipped with an orbifold locally symmetric metric
$\omega_S$ known as the Siegel metric.

Our second main result is the following (see Theorem \ref{thms}):

\begin{teo}\label{th2}
The $(1,1)$-form $\overline{\partial} (\beta^\eta)$ on $\mg$, where $\beta^\eta$ is the section
in Theorem \ref{th1}, is a nonzero constant scalar multiple
of $\tau^*\omega_S$, where $\omega_S$ is the Siegel metric on $\ag$ and $\tau$ is the Torelli map in \eqref{itau}.
\end{teo}

It may be mentioned that the second fundamental form of the Torelli map with respect to $\omega_S$ was
studied extensively in \cite{cpt}, \cite{cf}, \cite{cfg} (see also \cite{gpt}, \cite{fp}). Finally, we
prove the following characterization of the section $\widehat{\eta}$ in \eqref{i0} (see Theorem \ref{1,1})):

\begin{teo}\label{th3}
The section $\widehat\eta$ is the unique element of $H^0(S,\, \Omega^2_S(2 \Delta))$ with cohomology class
in $H^2(S\setminus \Delta,\,\bC)$ of pure type $(1,\,1)$ whose restriction to $\Delta$ is $1$.
\end{teo}

The paper is organized as follows. 
Section \ref{se5} does not contain any new material. In this section we recall 
the definition of projective structures on a Riemann surface, and collect the results on projective structures that 
are used here, the main one being the earlier mentioned theorem of Zograf and Takhtadzhyan from \cite{ZT}. This 
entails building the set-up which takes up the most of Section \ref{se5}.
In Section \ref{se2} we first recall the definition of $\widehat{\eta}$, and then 
give an interpretation of it that allows us to extend it in families (Section \ref{se2.3}). In Section \ref{se3} we 
give a proof of Theorem \ref{th1} (Theorem \ref{lem-ch} and Remark \ref{rgenus2}) and of Theorem \ref{th2} (Theorem 
\ref{thms}). In Section \ref{se4} we study the cohomology class of a meromorphic 2-form on a surface and we prove 
Theorem \ref{th3} (Theorem \ref{1,1}). 

\section{Projective structures on a Riemann surface}\label{se5}

Let $\mathbb V$ be a complex vector space of dimension two. Let ${\mathbb P}({\mathbb V})$
be the projective space parametrizing the lines in $\mathbb V$. The group
$\text{PGL}({\mathbb V})\,=\, \text{GL}({\mathbb V})/{\mathbb C}^*$ acts faithfully on 
${\mathbb P}({\mathbb V})$.

Let $C$ be a compact connected Riemann surface. A holomorphic coordinate chart on $C$ is a pair
of the form $(U,\, \phi)$, where $U\, \subset\, C$ is an open subset and $\phi\,
:\, U\, \longrightarrow\, {\mathbb P}({\mathbb V})$ is a holomorphic embedding. A holomorphic
coordinate atlas on $C$ is a collection of coordinate charts $\{(U_i,\, \phi_i)\}_{i\in I}$
such that $C\,=\, \bigcup_{i\in I} U_i$. A projective structure on $C$ is given by a
holomorphic coordinate atlas $\{(U_i,\, \phi_i)\}_{i\in I}$ such that for all
pairs $i, j\,\in\, I\times I$ for which $U_i\bigcap U_j\, \not=\, \emptyset$, there is
an element $\tau_{j,i}\, \in\, \text{PGL}({\mathbb V})$
such that $\phi_j\circ\phi^{-1}_i$ is the restriction of to $\phi_i(U_i\bigcap U_j)$
of the automorphism of ${\mathbb P}({\mathbb V})$ given by $\tau_{j,i}$.
Two such collections of pairs $\{(U_i,\, \phi_i)\}_{i\in I}$ and
$\{(U_i,\, \phi_i)\}_{i\in J}$ are called equivalent if
their union $\{(U_i,\, \phi_i)\}_{i\in I\cup J}$ is again a
part of a collection of pairs satisfying the above condition. A \textit{projective structure} on $C$
is an equivalence class of collection of pairs satisfying the above condition; see \cite{Gu}.
There are projective structures on $C$, for example, the uniformization of $C$ produces a
projective structure on $C$. In fact, the space of all projective structures on $C$ is
an affine space modelled on the complex vector space $H^0(C,\, K^{\otimes 2}_C)$ \cite{Gu}.

For $i\,=\,1,\, 2$, let $p_i\, :\, {\mathbb P}({\mathbb V})\times {\mathbb P}({\mathbb V})
\, \longrightarrow\, {\mathbb P}({\mathbb V})$ be the projection to the $i$-th factor. Consider
the holomorphic line bundle
$$
{\mathcal L}_0\, :=\, K_{{\mathbb P}({\mathbb V})\times{\mathbb P}({\mathbb V})}\otimes
{\mathcal O}_{{\mathbb P}({\mathbb V})\times{\mathbb P}({\mathbb V})}(2\Delta_0)\,=\,
(p^*_1K_{{\mathbb P}({\mathbb V})}\otimes p^*_2K_{{\mathbb P}({\mathbb V})})
\otimes {\mathcal O}_{{\mathbb P}({\mathbb V})\times{\mathbb P}({\mathbb V})}(2\Delta_0)
$$
over ${\mathbb P}({\mathbb V})\times{\mathbb P}({\mathbb V})$,
where $\Delta_0\, \subset\, {\mathbb P}({\mathbb V})\times{\mathbb P}({\mathbb V})$ is the 
reduced diagonal divisor. Since $\text{Pic}({\mathbb P}({\mathbb V})\times{\mathbb 
P}({\mathbb V})) \,=\, \text{Pic}({\mathbb P}({\mathbb V}))\oplus \text{Pic}({\mathbb 
P}({\mathbb V}))$, the line bundle ${\mathcal L}_0$ is trivializable. Also, using
the Poincar\'e adjunction formula it follows that ${\mathcal 
L}_0\vert_{\Delta_0}$ is canonically trivialized. Hence combining these it follows that 
${\mathcal L}_0$ is canonically trivialized. Let
\begin{equation}\label{e1}
\sigma_0\, \in\, H^0({\mathbb P}({\mathbb V})\times{\mathbb P}({\mathbb V}),\, {\mathcal L}_0)
\end{equation}
be the section giving the canonical trivialization. Note that the diagonal action of 
$\text{PGL}({\mathbb V})$ on ${\mathbb P}({\mathbb V})\times{\mathbb P}({\mathbb V})$ has a 
canonical lift to an action of $\text{PGL}({\mathbb V})$ on ${\mathcal L}_0$. The section 
$\sigma_0$ in \eqref{e1} is fixed by this action of $\text{PGL}({\mathbb V})$ on ${\mathcal 
L}_0$. The involution of ${\mathbb P}({\mathbb V})\times{\mathbb P}({\mathbb V})$ defined by 
$(x,\, y)\, \longmapsto\, (y,\, x)$ lifts canonically to an involution of ${\mathcal L}_0
\,=\, K_{{\mathbb P}({\mathbb V})\times{\mathbb P}({\mathbb V})}\otimes
{\mathcal O}_{{\mathbb P}({\mathbb V})\times{\mathbb P}({\mathbb V})}(2\Delta_0)$, 
and the meromorphic $2$--form $\sigma_0$ is mapped to $-\sigma_0$ by this involution of ${\mathcal L}_0$.

Let $C$ be a compact connected Riemann surface. Let
\begin{equation}\label{ppq}
p,\, q\, :\, C\times C\, \longrightarrow\, C
\end{equation}
be the projections to the first factor and second factor respectively. Let
\begin{equation}\label{dcl}
{\mathcal L}\, :=\, K_{C\times C}\otimes {\mathcal O}_{C\times C}(2\Delta)
\,=\, (p^*K_C\otimes q^*K_C)\otimes {\mathcal O}_{C\times C}(2\Delta)
\, \longrightarrow\, C\times C
\end{equation}
be the holomorphic line bundle on $C\times C$,
where $\Delta\, \subset\, C\times C$ is the reduced diagonal divisor.
Let
\begin{equation}\label{dnu}
\nu\, :\, C\times C\, \longrightarrow\, C\times C
\end{equation}
be the involution defined by
$(x,\, y)\, \longmapsto\, (y,\, x)$. This involution lifts canonically to an involution of
$\mathcal L$ defined by
\begin{equation}\label{il}
v_x\wedge w_y \, \longmapsto\, w_y\wedge v_x\, , \ \ v_x\, \in\, (K_C)_x\, ,\ \
w_y\, \in\, (K_C)_y\, , \ \ x,\, y\, \in\, C\, .
\end{equation}

For any $k\, \geq\, 1$, the $k$--th order infinitesimal neighborhood of $\Delta$ in
$C\times C$ will be denoted by $\Delta_{k+1}$.
Using the Poincar\'e adjunction formula, the restriction
${\mathcal L}\vert_\Delta$ is the trivial line bundle on $\Delta$. In fact, the trivialization
of ${\mathcal L}\vert_\Delta$ extends to a canonical trivialization of
${\mathcal L}\vert_{\Delta_2}$ (see \cite[p.~754, Theorem 2.1]{BR1}, \cite[p.~688, Theorem 2.2]{BR2}).
Let
\begin{equation}\label{evp}
\varpi_2\, \in\, H^0(\Delta_2,\, {\mathcal L}\vert_{\Delta_2})
\end{equation}
be the section giving this
canonical trivialization. This section $\varpi_2$ is uniquely determined
by the following two conditions:
\begin{itemize}
\item the restriction of $\varpi_2$ to $\Delta\, \subset\, \Delta_2$ coincides with the
section of ${\mathcal L}\vert_{\Delta}\,=\, {\mathcal O}_\Delta$ given by the constant function $1$, and

\item the involution in \eqref{il} takes $\varpi_2$ to $-\varpi_2$.
\end{itemize}

Let $P\,=\,\{(U_i,\, \phi_i)\}_{i\in I}$ be a projective structure on $C$.
For any $i\, \in\, I$, there is a natural isomorphism
$$
(\phi_i\times \phi_i)^*{\mathcal L}_0\, \stackrel{\sim}{\longrightarrow}\,
{\mathcal L}\vert_{U_i\times U_i}
$$
given by the differential of the map $\phi_i$. Using this isomorphism, the section
$\sigma_0$ in \eqref{e1} produces a section
$$
\sigma_{0,i}\,=\,(\phi_i\times \phi_i)^*\sigma_0
\, \in\, H^0(U_i\times U_i,\, {\mathcal L}\vert_{U_i\times U_i})\, .
$$
Since $\sigma_0$ is fixed by the action of $\text{PGL}({\mathbb V})$ on ${\mathcal L}_0$, these
sections $\sigma_{0,i}$ patch together compatibly to define a section
\begin{equation}\label{e3}
\widetilde{\sigma}\, \in\, H^0({\mathcal U},\, {\mathcal L}\vert_{\mathcal U})\, ,
\end{equation}
where ${\mathcal U}\, \subset\, C\times C$ is an analytic open subset containing the diagonal $\Delta$.

The section $\widetilde{\sigma}$ in \eqref{e3} has the following two properties:
\begin{enumerate}
\item The involution of $\mathcal L$ in \eqref{il} lifting $\nu$ in \eqref{dnu} takes $\widetilde\sigma$
to $-\widetilde\sigma$, and

\item the restriction $\widetilde{\sigma}\vert_{\Delta_2}$ coincides with the canonical
trivialization of ${\mathcal L}\vert_{\Delta_2}$ in \eqref{evp}; this follows from \cite[p.~756, Proposition
2.10]{BR1}. Note that $\widetilde{\sigma}\vert_{\Delta}$ coincides with the canonical
trivialization of ${\mathcal L}\vert_{\Delta}$, and hence invoking the earlier mentioned characterization of
the canonical trivialization of ${\mathcal L}\vert_{\Delta_2}$ if follows that $\widetilde{\sigma}\vert_{\Delta_2}$
coincides with it.
\end{enumerate}

Let
$$
{\mathcal Q}(C)\, \subset\, H^0(\Delta_3,\, {\mathcal L}\vert_{\Delta_3})
$$
be the locus of all sections $s$ such that the restriction $s\vert_{\Delta_2}$ coincides with the
canonical trivialization of ${\mathcal L}\vert_{\Delta_2}$ in \eqref{evp}. This ${\mathcal Q}(C)$ is
evidently an affine space modelled on the complex vector space $H^0(C,\, K^{\otimes 2}_C)$. Let
\begin{equation}\label{e4}
\sigma\, \in\, H^0(\Delta_3,\, {\mathcal L}\vert_{\Delta_3})
\end{equation}
be the restriction of the section $\widetilde\sigma$ in \eqref{e3}
to the nonreduced divisor $\Delta_3$. Since the restriction $\sigma\vert_{\Delta_2}$ coincides with the
canonical trivialization of ${\mathcal L}\vert_{\Delta_2}$ in \eqref{evp},
it follows that $\sigma\, \in\, {\mathcal Q}(C)$.

Let ${\mathcal P}(C)$ denote the space of all projective structures on $C$. We have a map
\begin{equation}\label{ephi}
\Phi\, :\, {\mathcal P}(C)\, \longrightarrow\, {\mathcal Q}(C)
\end{equation}
that sends any $P\, \in\, {\mathcal P}(C)$ to $\sigma\, \in\, {\mathcal Q}(C)$ constructed in
\eqref{e4} using $P$. This map $\Phi$ is an isomorphism of affine spaces modelled on 
$H^0(C,\, K^{\otimes 2}_C)$ (see \cite[p.~757, Theorem 3.2]{BR1} and \cite[p.~758, Lemma 3.6]{BR1};
see also \cite[p.~688, Theorem 2.2]{BR2}).

Let
\begin{equation}\label{elpi}
\pi\, :\, {\mathcal C}\, \longrightarrow\, B
\end{equation}
be a smooth holomorphic family of irreducible complex projective curves of genus $g$. Let
\begin{equation}\label{dpqt}
\widetilde{p},\, \widetilde{q}\, :\, {\mathcal C}\times_B{\mathcal C}\, \longrightarrow\,\mathcal C
\end{equation}
be the projections to the first factor and second factor respectively. Let
\begin{equation}\label{edb}
{\bf\Delta}_B\,:=\,\{(x,\,x) \ \mid \ x \,\in\, {\mathcal C}\} \,\subset\, {\mathcal C}\times_B{\mathcal C}
\end{equation}
be the reduced relative diagonal divisor in ${\mathcal C}\times_B{\mathcal C}$.
Let ${\mathcal K}\, \longrightarrow\, \mathcal C$ be the relative canonical bundle for the
projection $\pi$ in \eqref{elpi}. Consider the family of surfaces
\begin{equation}\label{Pi}
\Pi\, :\, {\mathcal C}\times_B{\mathcal C}\, \longrightarrow \,B\, .
\end{equation}
Let
\begin{equation}\label{mbl}
{\mathbb L}\, :=\, K_{({\mathcal C}\times_B{\mathcal C})/B}\otimes
{\mathcal O}_{{\mathcal C}\times_B{\mathcal C}}(2{\bf\Delta}_B)\,=\,
({\widetilde p}^*{\mathcal K}\otimes {\widetilde q}^*{\mathcal K})\otimes
{\mathcal O}_{{\mathcal C}\times_B{\mathcal C}}(2{\bf\Delta}_B)
\end{equation}
be the holomorphic line bundle on this family ${\mathcal C}\times_B{\mathcal C}$,
where $\widetilde p$ and $\widetilde q$ are the projections in \eqref{dpqt}.

Since $k{\bf\Delta}_B$ is an effective divisor on ${\mathcal C}\times_B {\mathcal C}$
for any positive integer $k$, we have
$$
{\mathcal O}_{{\mathcal C}\times_B {\mathcal C}}(-k{\bf\Delta}_B)\, \subset\,
{\mathcal O}_{{\mathcal C}\times_B {\mathcal C}}\, .
$$
Now tensoring the above inclusion with the identity map of ${\mathbb L}$ we have
$$
{\mathcal O}_{{\mathcal C}\times_B {\mathcal C}}(-k{\bf\Delta}_B)\otimes {\mathbb L}\, \subset\,
{\mathcal O}_{{\mathcal C}\times_B {\mathcal C}}\otimes {\mathbb L}\,=\, {\mathbb L}\, .
$$
Hence the quotient for the above inclusion
$$
{\mathbb L}/({\mathcal O}_{{\mathcal C}\times_B
{\mathcal C}}(-k{\bf\Delta}_B)\otimes {\mathbb L})
$$
is a coherent analytic sheaf on ${\mathcal C}\times_B {\mathcal C}$; it is supported on the
$(k-1)$-th order infinitesimal neighborhood of ${\bf\Delta}_B\, \subset\, {\mathcal C}\times_B {\mathcal C}$.
Note that ${\mathbb L}/({\mathcal O}_{{\mathcal C}\times_B
{\mathcal C}}(-k{\bf\Delta}_B)\otimes {\mathbb L})$ coincides with $I(k)_*I(k)^*{\mathbb L}$, where
$I(k)$ is the inclusion map of the $(k-1)$-th order infinitesimal neighborhood of ${\bf\Delta}_B$
to ${\mathcal C}\times_B {\mathcal C}$.
Now using the map $\Pi$ in \eqref{Pi}, construct the direct images 
\begin{equation}\label{cvn}
{\mathcal V}\, :=\, \Pi_*({\mathbb L}/({\mathcal O}_{{\mathcal C}\times_B
{\mathcal C}}(-3{\bf\Delta}_B)\otimes {\mathbb L}))\,\longrightarrow\,B
\end{equation}
and
$$
{\mathcal V}_2\, :=\, \Pi_*({\mathbb L}/({\mathcal O}_{{\mathcal C}\times_B
{\mathcal C}}(-2{\bf\Delta}_B)\otimes {\mathbb L}))\,\longrightarrow\, B
$$
which are holomorphic vector bundles over $B$. Since
$${\mathbb L}/({\mathcal O}_{{\mathcal C}\times_B
{\mathcal C}}(-k{\bf\Delta}_B)\otimes {\mathbb L})\,=\, I(k)_*I(k)^*{\mathbb L}\, ,$$ if $\Delta^b\,\subset\,
\pi^{-1}(b)\times \pi^{-1}(b)$ is the reduced diagonal for any point $b\, \in\, B$, then
the fiber of ${\mathcal V}$ (respectively, ${\mathcal V}_2$) over $b$ is
$H^0(\Delta^b_3,\, {\mathbb L}\vert_{\Delta^b_3})$ (respectively, $H^0(\Delta^b_2,\, {\mathbb L}\vert_{\Delta^b_2})$),
where $\Delta^b_k$ is the $(k-1)$-th order infinitesimal neighborhood of $\Delta^b$ in
the surface $\pi^{-1}(b)\times \pi^{-1}(b)$. 

There is a natural surjective homomorphism
\begin{equation}\label{psi}
\Psi\, :\, {\mathcal V}\, \longrightarrow\, {\mathcal V}_2
\end{equation}
given by projection ${\mathbb L}/({\mathcal O}_{{\mathcal C}\times_B
{\mathcal C}}(-3{\bf\Delta}_B)\otimes {\mathbb L})\, \longrightarrow\,
{\mathbb L}/({\mathcal O}_{{\mathcal C}\times_B
{\mathcal C}}(-2{\bf\Delta}_B)\otimes {\mathbb L})$ constructed using
the inclusion map $${\mathcal O}_{{\mathcal C}\times_B
{\mathcal C}}(-2{\bf\Delta}_B)\otimes {\mathbb L})\, \hookrightarrow\,
{\mathcal O}_{{\mathcal C}\times_B
{\mathcal C}}(-3{\bf\Delta}_B)\otimes {\mathbb L})\, .$$ The vector bundle
${\mathcal V}_2$ has a tautological holomorphic section given by the
canonical trivialization of ${\mathcal L}\vert_{\Delta_2}$ in \eqref{evp} for any curve $C$.
This holomorphic section of ${\mathcal V}_2$ will be denoted by $s_0$. Now define
\begin{equation}\label{whcv}
\widehat{\mathcal V}\, :=\, \Psi^{-1}(s_0)\, \subset\, \mathcal V\, ,
\end{equation}
where $\Psi$ is the projection in \eqref{psi}. We note that $\widehat{\mathcal V}$ is a holomorphic
affine bundle over $B$ modelled on the vector bundle $\pi_*{\mathcal K}^{\otimes 2}$, where $\pi$
is the projection in \eqref{elpi} and
$\mathcal K$ is the relative canonical bundle in \eqref{mbl}. Indeed, this 
follows immediately from the fact that
\begin{equation}\label{ke}
\text{kernel}(\Psi)\,=\, \pi_*{\mathcal K}^{\otimes 2}\, .
\end{equation}

The following lemma is an immediate consequence of the isomorphism $\Phi$ in \eqref{ephi}.

\begin{lemma}\label{lem1}
The $C^\infty$ (respectively, holomorphic) sections of the fiber bundle $\widehat{\mathcal V}
\, \longrightarrow\,B$ in \eqref{whcv} are in a natural bijective correspondence with
the $C^\infty$ (respectively, holomorphic) families of projective structures for the family of
curves $\mathcal C$.
\end{lemma}

Let
$$
\beta\, :\, B\, \longrightarrow\, \widehat{\mathcal V}
$$
be a $C^\infty$ section. Denote the Dolbeault operator for the holomorphic vector bundle 
$\mathcal V$, defined in \eqref{cvn}, by $\overline{\partial}_{\mathcal V}$. Since $\beta$ 
is also a section of $\mathcal V$, we have
$$
\overline{\partial}_{\mathcal V}(\beta)\, \in\, \Omega^{0,1}(B, \,{\mathcal V})\, .
$$
The projection $\Psi$ in \eqref{psi} is evidently holomorphic; recall that
$\Psi(\beta)$ is a holomorphic section of ${\mathcal V}_2$. These imply that for the section
$\Psi\circ\overline{\partial}_{\mathcal V}(\beta)\, \in\, \Omega^{0,1}(B, \,{\mathcal V}_2)$, we have
$$
\Psi\circ\overline{\partial}_{\mathcal V}(\beta)\,=\, \overline{\partial}_{{\mathcal V}_2}
(\Psi(\beta))\, =\, 0\, ,
$$
where $\overline{\partial}_{{\mathcal V}_2}$ is the Dolbeault operator for the holomorphic
vector bundle ${\mathcal V}_2$. Hence from \eqref{ke} it follows immediately that
\begin{equation}\label{k2}
\overline{\partial}_{\mathcal V}(\beta)\, \in\, \Omega^{0,1}(B, \,
\pi_*{\mathcal K}^{\otimes 2})\, .
\end{equation}

Let $\mg$ denote the moduli space of smooth complex
projective curves of genus $g$. Let
\begin{equation}\label{ega}
\Gamma\, :\, B\, \longrightarrow\, {\mathsf M}_g
\end{equation}
be the holomorphic map to the moduli space of curves corresponding to the above family $\mathcal C$ over
$B$.

Let $\omega_{wp}$ be the Weil--Petersson K\"ahler form on ${\mathsf M}_g$; it
is actually an orbifold $(1,1)$--form. The projective structure on the Riemann surfaces given by
the uniformization theorem produces a $C^\infty$ section of $\widehat{\mathcal V}$. Let
\begin{equation}\label{bu}
\beta^u\, :\, B\, \longrightarrow\, \widehat{\mathcal V}
\end{equation}
be the $C^\infty$ section given by
the uniformization theorem. For this $\beta^u$ it was shown by Zograf and Takhtadzhyan that
\begin{equation}\label{wp}
\overline{\partial}_{\mathcal V}(\beta^u)\, =\, \Gamma^*\omega_{wp}\,\in\,
\Omega^{0,1}(B, \, \Gamma^* \Omega^{1,0}_{{\mathsf M}_g})
\end{equation}
\cite[p.~310, Theorem 2]{ZT}, \cite[p.~311, Remark 3]{ZT}; see \cite[p.~214, Theorem 1.7]{Iv}
for an excellent exposition (see also \cite[p.~355, Theorem 9.2]{Mc}).

Let $\beta_1,\, \beta_2\, :\, B\,\longrightarrow\,\widehat{\mathcal V}$ be two
$C^\infty$ sections of $\widehat{\mathcal V}\, \longrightarrow\,B$ such that 
$$
\overline{\partial}_{\mathcal V}(\beta_1)\, =\,
\overline{\partial}_{\mathcal V}(\beta_2)\, .
$$
Since $\Psi(\beta_1-\beta_2)\,=\, 0$, where $\Psi$ is the projection in \eqref{psi}, from
\eqref{ke} we conclude that
\begin{equation}\label{c1}
\beta_1-\beta_2\, \in\, H^0(B,\, \Gamma^* \Omega^{1,0}_{{\mathsf M}_g})\, ,
\end{equation}
where $\Gamma$ is the map in \eqref{ega}.

\begin{lemma}\label{lem2}
Let $\beta_1,\, \beta_2\, :\, B\,\longrightarrow\,\widehat{\mathcal V}$ be two
$C^\infty$ sections of $\widehat{\mathcal V}\, \longrightarrow\,B$ such that 
$$
\overline{\partial}_{\mathcal V}(\beta_1)\, =\,
\overline{\partial}_{\mathcal V}(\beta_2)\, .
$$
Then $\beta_1\,=\, {\mathbb T}\circ \beta_2$, where ${\mathbb T}$ is a holomorphic 
automorphism of the $\Gamma^* \Omega^{1,0}_{{\mathsf M}_g}$--torsor $\widehat{\mathcal V}$ 
over $B$. Conversely, for any holomorphic automorphism $\mathbb T$ of the $\Gamma^* 
\Omega^{1,0}_{{\mathsf M}_g}$--torsor $\widehat{\mathcal V}$, and for any $C^\infty$ section 
$\beta\, :\, B\,\longrightarrow\,\widehat{\mathcal V}$ of\, $\widehat{\mathcal V}\, 
\longrightarrow\,B$,
$$
\overline{\partial}_{\mathcal V}(\beta)\, =\,
\overline{\partial}_{\mathcal V}({\mathbb T}\circ\beta)\, .
$$
\end{lemma}

\begin{proof}
This follows from \eqref{c1}, because $H^0(B,\, \Gamma^* \Omega^{1,0}_{{\mathsf M}_g})$ is 
in fact the group of holomorphic automorphisms of the $\Gamma^* \Omega^{1,0}_{{\mathsf 
M}_g}$--torsor $\widehat{\mathcal V}$ over $B$.
\end{proof}

\section{The intrinsic meromorphic 2-form on $C\times C$}\label{se2}

In this section we first recall an intrinsic meromorphic 2-form $\widehat{\eta}$ on $C\times C$, 
constructed in \cite{cfg} using results from \cite{cpt}, which governs the second fundamental form of the Torelli map 
{\cite[Theorem 3.7]{cfg}}.
We then give a new interpretation of this form in Subsection \ref{se2.2}, which allows us to extend it
on families (this is done in Subsection \ref{se2.3}).
This extension connects $\widehat\eta$ with the the Hodge decomposition, as shown in Theorem \ref{splitting}.

\subsection{Construction of $\widehat{\eta}$}\label{se2.1}

In this subsection we recall the definition and some properties of the form $\widehat{\eta}$ following \cite{cfg}. 

The canonical line bundle of a complex manifold $Y$ will be denoted by $K_Y$.

As before, $C$ is a smooth complex projective curve; take a point $x \,\in\, C$. Let
\begin{equation}\label{jx}
j_x\,: \, H^0(C,\,K_C(2x))
 \,\hookrightarrow\, H^1(C\setminus\{x\} ,\,\mathbb{C})\,=\,H^1(C,\,\mathbb{C})
\end{equation}
be the injective homomorphism that associates to a meromorphic $1$--form, with at most a double 
pole at $x$, its de Rham cohomology class. Since
$$\dim H^0(C,K_C(2x))\,=\,g+1\ \ \text{ and }\ \ H^{1,0}(C) \,\subset\, j_x(H^0(C,\, K_C(2x)))\, ,$$
the inverse image 
$j^{-1}_x(H^{0,1}(C))$ has dimension $1$; here $H^{0,1}(C)$ is considered as a subspace
of $H^1(C,\,\mathbb{C})$ using the Hodge
decomposition. If we fix a local holomorphic coordinate function $z$ on 
a neighborhood of $x$ with $z(x)\,=\,0$, there exists a unique element $\mu$ in this line 
$j^{-1}_x(H^{0,1}(C))$ whose expression on $U$ is
\begin{equation}
\label{phi}
\mu\vert_U \,:=\, \bigg( \frac{1}{z^2}+h(z) \bigg) dz\, ,
\end{equation}
where $h$ is a holomorphic function. So we have a map
\begin{equation}\label{etax}
\eta_x\,:\, T_xC \,\longrightarrow\, H^0(C,\, K_C(2x))
\end{equation}
that sends $\lambda \frac{\partial}{\partial z}(x)$ to $\lambda \mu$; this
map is evidently independent of the choice of the holomorphic coordinate function $z$.

The following is proved in \cite{cfg}.

\begin{lemma}[{\cite[Lemma 3.5]{cfg}}]\label{eta-can}
Identify $H^{0,1}(C)$ with $H^0(C,\,K_C)^*$ using Serre
duality. Then the line $$j_x(H^0(C,\,K_C(2x)))\cap H^{0,1}(C)\, \subset\,
H^{0,1}(C)\, ,$$
where $j_x$ is the map in \eqref{jx}, corresponds to the hyperplane
in $H^0(C,\,K_C)$ defined by all $1$-forms vanishing at $x$.
\end{lemma}

Consider the complex surface $S\,:= \,C\times C$ and the line bundle
${\mathcal L}\,:=\, K_S\otimes{\mathcal O}_S(2
\Delta)\,=\, K_S(2\Delta)$ constructed on it in \eqref{dcl}. Let
$$
V\,:=\, p_*((q^*K_C)\otimes{\mathcal O}_S(2\Delta)) \ \ \text{ and }\ \ E\,:=\, p_* {\mathcal L}
$$
be the vector bundles $C$, where $p$ and $q$ are the projections in \eqref{ppq}.
The projection formula, \cite[p.~426, A4]{Ha}, says that
$E \,=\, K_C \otimes V$. Since $$(q^*K_C\otimes{\mathcal O}_S(2
\Delta))\restr{\{x\}\times C} \,=\, K_C (2x)\, ,$$ we have
$$H^0(p^{-1}(x) ,\, (q^*K_C)\otimes{\mathcal O}_S(2\Delta)\vert_{p^{-1}(x)})
\,=\, H^0 (C,\, K_C (2x))\, .$$
The fiber of the holomorphic vector bundle $V\, \longrightarrow\, C$ over $x\,\in\, C$ is $H^0(C ,
\, K_C(2x))$, and the map $x\,\longmapsto\, \pet_x$ (constructed in \eqref{etax}) is a 
$C^\infty$ section of $K_C\otimes V\,=\, E$. Let
\begin{equation}\label{eeta}
\eta\, \in\, C^\infty(C;\, E)
\end{equation}
be this section given by $x\,\longmapsto\, \pet_x$.

The following was proved in \cite{cfg}.

\begin{prop}[{\cite[Proposition 3.4]{cfg}}]\label{prophi}
The section $\pet$ in \eqref{eeta} is in fact holomorphic. 
\end{prop}

Since $E\,=\,p_* {\mathcal L}$, there is an isomorphism $H^0(C,\,E) \,\cong\, H^0(S,\,{\mathcal L})$
that associates to any $\alfa \,\in \,H^0(C,\,E)$ the section $\widehat{\alf} \,\in\,
H^0(S,\,{\mathcal L})$ such that
$$
\alf_x\,=\, \widehat{\alf} \restr{\{x\}\times C} \,\in\, T_x^*C \otimes
H^0(C,\, K_C(2x)) \,=\, E_x\, .
$$
Thus, from Proposition \ref{prophi} it follows that there is a holomorphic section
\begin{equation}\label{weta}
\teta \,\in\, H^0(S,\,{\mathcal L})
\end{equation}
corresponding to $\pet$.

\begin{prop}[{\cite[Lemma 3.5]{cfg}}]\label{prophi2}
The tautological lift to $\mathcal L$ of the involution of $C\times C$, defined by
$(x,\, y)\, \longmapsto\, (y,\, x)$, takes the section $\teta$ in \eqref{weta} to $-\teta$.
\end{prop}

Let $z$ be a local holomorphic coordinate on on $C$; denote $z\circ p$
and $z\circ q$ by $z_1$ and $z_2$ respectively, where $p$ and $q$ are the projections in \eqref{ppq}.
Locally, around the diagonal, we have
$$
\widehat{\eta}\,=\, \frac{dz_1\wedge dz_2}{(z_1- z_2)^2}+ f(z_1,
z_2)dz_1\wedge dz_2\, ,
$$
where $f$ is a holomorphic
function with $f(z_1,z_2)\,=\,f(z_2,z_1)$ (see Proposition \ref{prophi2}).
The form $\widehat{\eta}$ also appears in an unpublished book of Gunning; he
calls it the ``intrinsic double differential of the second kind'' \cite{gun}. 

\subsection{A new interpretation of $\eta$}\label{se2.2}

Consider the differential
\begin{equation}\label{epsi}
\psi\,:\, TC \,\longrightarrow\, H^1(C,\, \OO_C) \otimes \OO_C
\end{equation}
of the Abel-Jacobi map $$C\, \longrightarrow\, \text{Pic}^1(C)\, ,\ \
x\, \longmapsto\, {\mathcal O}_C(x)\, .$$
It is the dual of the evaluation map $ev\,:\, H^0(C,\, K_C) \otimes \OO_C\,\longrightarrow\, K_C$,
after we identify $H^1(C,\, {\mathcal O}_C)^*$ with $H^0(C, \, K_C)$ using
Serre duality. 

Let $i_C\,:\, H^1(C,\OO_C)\,\hookrightarrow\,H^1(C,\, \C) $ be the natural identification of $H^1(C,\, \OO_C)$ with
$H^{0,1}(C)$ combined with the Hodge decomposition. The maps $j_x$ and $\eta_x$ below are
constructed in \eqref{jx} and \eqref{etax} respectively.

\begin{lemma}
The following diagram is commutative:
\begin{equation}\label{cd}
\begin{tikzcd}[column sep=large] 
{} & H^1(C,\,\OO_C) \arrow[hook]{rd}{i_C} & \\
T_x C \arrow{ru}{\psi_x} \arrow{rd}[swap]{-\pet_x} & & H^1(C,\, \C) \\
& H^0(C,\,K_C(2x)) = V_x \arrow[hook]{ru}{j_x} &
\end{tikzcd}
\end{equation}
\end{lemma}

\begin{proof}
The diagram \eqref{cd} is indeed exactly (3.1) 
in \cite[p.~9]{cfg}, because $\psi_x(u)$, $u\, \in\, T_xC$, is the element of $H^0(C,\, 
K_C)^*\,=\, H^1(C,\, \OO_C)$ defined by $\om \,\longmapsto\, \om_x(u)$ for all $\om\, \in\, 
H^0(C,\, K_C)$.
\end{proof}

We have $\eta \,\in \, H^0(C,\,E)$ by Proposition
\ref{prophi}, and hence $\eta$ corresponds to a holomorphic homomorphism 
\begin{equation}\label{eeta2}
\eta\,:\, T_C \,\longrightarrow \, p_*((q^*K_C)\otimes{\mathcal O}_S(2\Delta))
\,=\, p_*((q^*K_C)(2\Delta))\, .
\end{equation}

\begin{prop}
There is a holomorphic homomorphism 
$$j: p_* (q^*K_C(2 \Delta)) \hookrightarrow H^1(C,\, \OO_C) \otimes \OO_C\, ,$$ which is equal to $j_x$ at every
point $x \,\in \,C$, such that the following diagram is commutative
\begin{equation}\label{cd2}
\begin{tikzcd}[column sep=large] 
{} & H^1(C, \,\OO_C)\otimes \OO_C \arrow[hook]{rd}{i} & \\
T_C \arrow[hook]{ru}{\psi} \arrow{rd}[swap]{-\pet} & & H^1(C,\, \C) \otimes \OO_C \\
& p_* (q^*K_C(2 \Delta)) \arrow[hook]{ru}{j} &
\end{tikzcd}
\end{equation}
where $\eta$ is the homomorphism in \eqref{eeta2}, and $p,\, q$ are the projections in
\eqref{ppq}.
\end{prop}

\begin{proof} 
Let $$d\,:\, \OO_S(\Delta)\,
\longrightarrow\, \Omega^1_S\otimes{\mathcal O}_S(2 \Delta) \,=\, p^*K_C(2\Delta) \oplus q^*K_C(2
\Delta)$$ be the de Rham differential. Let
$$T\,:\, \OO_S(\Delta) \,\longrightarrow \,q^*K_C(2 \Delta)$$
be the composition of this homomorphism $d$ with the projection of
$p^*K_C(2\Delta) \oplus q^*K_C(2 \Delta)$ to the second factor.

We now prove the following :

{\bf Claim}.\,\, The kernel of $T$ is the subsheaf $p^{-1}(\OO_C)$. 

It suffices to prove this in terms of local holomorphic coordinates on the curve.
Let $z$ be a locally defined holomorphic coordinate function on $C$; as before, denote $z\circ p$
and $z\circ q$ by $z_1$ and $z_2$ respectively.
Set $$x \,= \,z_2-z_1\,,\ \ y \,=\, z_1 + z_2\, ,$$ and take a local section 
$f \,= \,\frac{a(y)}{x} + b(x,y)$ of $\OO_S(\Delta)$,
where $b$ is holomorphic. We have $$T(f) \,=\,
\frac{\partial f}{\partial z_2} dz_2 \,=\, \left(-\frac{a(y)}{x^2} +
\frac{a'(y)}{x}\right) dz_2 + \frac{\partial b}{\partial z_2}dz_2\, ,$$ and hence
$T(f) \,=\, 0$ if and only if the
following two hold:
\begin{itemize}
\item $a(y)\,\equiv\, 0$, and

\item $\frac{\partial b}{\partial z_2} \,=\,0$.
\end{itemize}
Therefore, $T(f)\,=\, 0$ if and only if $f \,=\, b(z_1)$, which proves the claim. \qed

Consequently, we have the short exact sequence:
\begin{equation}\label{j}
0\,\longrightarrow\, p^{-1}(\OO_C)\,\longrightarrow\,\OO_S(\Delta)\,\stackrel{T}{\longrightarrow}
\, T(\OO_S(\Delta))\, =:\, {\mathcal E} \,\longrightarrow\, 0\, .
\end{equation}

The sheaf ${\mathcal E}$ in \eqref{j} is the kernel of the homomorphism
$$r\,:\, q^*K_C(2 \Delta)\,\longrightarrow\, p^{-1}(\OO_C)$$ constructed as follows:
Let $U \,\subset\, S$ be an open subset
and $\omega \,\in\, H^0(U,\, q^*K_C(2 \Delta))$. If $U\bigcap\Delta\,=\, \emptyset$,
then set $r(\omega)\,=\,0$. If $(z_1,\,z_1) \,\in \,\Delta\bigcap U$, set
$$r(\omega)(z_1) \,:= \,\int_{\gamma_{z_1}} \omega\, ,$$ where $\gamma_{z_1}$
is a small oriented circle around $z_1$ in the fiber $(\{z_1\}\times C)\bigcap
U$. In local holomorphic coordinates, assuming $D_{2 \epsilon} \times D_{2 \epsilon}
\,\subset\, U$, we have $$r(\omega)(z_1) = \int_{|z_2-z_1| \,=\, \epsilon}
\omega\, .$$ If $\omega \,=\, (\frac{a(y)}{x^2} + \frac{b(y)}{x} + c(x,y))
dz_2$, we have
$$
r(\omega)(z_1) \,=\, \int_{|z_2-z_1| \,=\, \epsilon} \left(\frac{a(z_1 + z_2)}{(z_2 - z_1)^2} +
\frac{b(z_1 + z_2)}{z_2 - z_1} + c\right) dz_2
$$
$$
=\, \int_{|z_2-z_1| \,=\, \epsilon} \left(\frac{a(2z_1) +(z_2-z_1)( a'(2z_1) +b(2z_1))}{(z_2-z_1)^2}
 + {\mathcal H}\right) dz_2\,=\, 2 \pi\sqrt{-1} (a'(2z_1) +b(2z_1))\, ,
$$
where $\mathcal H$ is holomorphic.
So $r(\omega) \,\in\, \OO(D_{2 \epsilon})$ and $\text{kernel}(r) \,=\, {\mathcal E}$. 

We will show that
\begin{equation}\label{esh}
p_*{\mathcal E} \,\cong\, p_* (q^*K_C(2 \Delta))
\end{equation}
To prove \eqref{esh}, take an open subset $U \,\subset\, C$ biholomorphic to the disk. To prove that $p_*{\mathcal
E} (U) \,=\, p_* q^* K_C(2\Delta) (U)$, it suffices to show that if $$\om \,\in\,
H^0(U,\, p_* q^* K_C(2\Delta)) \,=\, H^0(U\times C,\, q^*K_C(2\Delta))\, ,$$ then $\om
\,\in\, H^0(U\times C,\, \mathcal{E})$.

Now, if $x \,\in\, U$ is fixed, then $\om (x,\, \cd) \,\in\, H^0(C,\, K_C(2x))$ does not
have residue at $x$, and hence $r(\om)(x) \,= \, 0$. From this it follows that $\om\,\in\,
H^0(U\times C,\, \mathcal{E})$. Hence \eqref{esh} is proved.

Note that $p^{-1}(\OO_C) \,=\, p^{-1}(\OO_C) \otimes_{\C} q^{-1}(\C_{C})$, and by K\"unneth's
formula, \cite[p. 244]{Demailly}, for an open subset $U \,\subset\, C$,
$$H^1(U \times C,\, p^{-1}(\OO_C) \otimes_{\C} q^{-1} (\C_{C})) \,\cong $$
$$(H^0(U,\,\OO_C) \otimes H^1(C,\, \C)) \oplus (H^1(U,\, \OO_C) \otimes H^0(C,\, \C))
\,=\, \OO_C(U) \otimes H^1(C, \,\C)\, ,$$ because $H^1(U,\, \OO_C) \,=\,0$. Consequently, we have
\begin{equation}\label{esh1}
R^1p_*(p^{-1}\OO_C)(U) \,=\, \OO_C(U) \otimes H^1(C,\, \C)\, .
\end{equation}
Using the same method we get that
\begin{equation}\label{esh2}
R^2p_* p^{-1}(\OO_C) \,= \, \OO_C \otimes H^2(C,\,\C) \,= \,\OO_C\, .
\end{equation}
So, applying $p_* $ to the short exact sequence
$$
0\,\longrightarrow\, {\mathcal E}\,\longrightarrow\,
q^*K_C(2\Delta) \,\stackrel{r}{\longrightarrow}\, p^{-1}(\OO_C) \,\longrightarrow\, 0
$$
we conclude that $R^1p_* {\mathcal E}\,\cong\, \OO_C$.

Now we apply $p_*$ to the short exact sequence \eqref{j}. Since $$p_*(p^{-1}(\OO_C))\,=\, p_*
\OO_S(\Delta) \,=\, \OO_C\, ,$$ the following exact sequence is obtained:
\begin{equation}\label{successione2}
0\,\longrightarrow\, p_*{\mathcal E} \,\cong\,p_* (q^*K_C(2 \Delta))\,\stackrel{j}{\longrightarrow}
\, R^1p_*(p^{-1}(\OO_C)) \,\longrightarrow\, R^1p_*(\OO_S(\Delta))
\end{equation}
$$
\longrightarrow\, R^1p_* {\mathcal E} \,\cong\, \OO_C \,\longrightarrow\,
R^2 p_*(p^{-1} (\OO_C))\,\longrightarrow\,0\, .
$$
In view of \eqref{esh1} and \eqref{esh2}, the exact sequence in \eqref{successione2} becomes
$$
0 \,\longrightarrow\,V\,\stackrel{j}{\longrightarrow}\, H^1(C, \,\C) \otimes \OO_C
\,\longrightarrow\,R^1 p_*\OO_S(\Delta) \,\longrightarrow\, 0\, .
$$
Notice that by construction the above homomorphism $j$ at any point $x \,\in\, C$ is the map $j_x$ in
\eqref{cd}. The commutativity of the diagram in \eqref{cd2} follows from the commutativity of
the diagram in \eqref{cd}. 
\end{proof}

\subsection{The relative version}\label{se2.3}

Recall the family of smooth curves $\pi\,:\, {\mathcal C}\,\longrightarrow\, B$ in \eqref{elpi}.
Define the fiber product
\begin{equation}\label{ecs}
{\mathcal S} \,: =\, {\mathcal C} \times_B {\mathcal C}\, ,
\end{equation}
Now consider the map
\begin{equation}\label{psirel}
0 \, \longrightarrow\, T_{\Ci/B} \,\stackrel{\psi}{\longrightarrow}\, R^1\widetilde{p}_* \OO_{\esse}
\end{equation}
which is the dual of the evaluation map $\widetilde{p}_* \widetilde{q}^* K_{\Ci/B}
\, \longrightarrow\, K_{\Ci/B}$ (the maps $\widetilde p$ and $\widetilde q$ are defined in
\eqref{dpqt}).
The restriction of this map to the fiber over any point
of $B$ coincides with $\psi$ in \eqref{cd2} (this re-use of notation should not cause
any confusion).

Consider the variation of Hodge structure for the family
$\widetilde{p}$ in \eqref{dpqt}:
\begin{equation}\label{hes}
0 \,\longrightarrow\, \widetilde{p}_*\widetilde{q}^*K_{{\mathcal C}/B} \,\longrightarrow\,
R^1 \widetilde{p}_* \C_{\esse} \otimes \OO_{\Ci}\,\longrightarrow\, R^1\widetilde{p}_*\OO_{\esse}
\,\longrightarrow\, 0\, ,
\end{equation}
where $\mathcal S$ is defined in \eqref{ecs}. Let
\begin{equation}\label{fi}
i\,:\, R^1\widetilde{p}_* \OO_{\esse}\,\longrightarrow\,
R^1 \widetilde{p}_* \C_{\esse} \otimes \OO_{\Ci}
\end{equation}
be the $C^{\infty}$ splitting of it given by the Hodge decomposition. 
The homomorphism in \eqref{fi} is the relative version of the map $i$ in \eqref{cd2}.

We denote by
\begin{equation}\label{releta}
\eta\,:\, T_{\Ci/B}\, \longrightarrow\, \widetilde{p}_*(\widetilde{q}^* K_{\Ci/B}(2 \bf\Delta_B))
\end{equation}
the ($C^{\infty}$) relative version of the map $\eta$ in \eqref{cd2}, where
$\bf\Delta_B$ is defined in \eqref{edb}.

We have the following:

\begin{prop}\label{diagram-eta}
There is a relative version of the map $j$ in \eqref{cd2},
$$j\,:\, \widetilde{p}_*(\widetilde{q}^* K_{\Ci/B}(2 {\bf{\Delta_B}}))
\,\longrightarrow\, R^1 \widetilde{p}_* \C_{\esse} \otimes \OO_{\Ci} $$ such that 
the diagram
$$
 \begin{tikzcd}[column sep=large] 
 \label{diagramma}
 {} & R^1\widetilde{p}_* \OO_{\esse} \arrow[hook]{rd}{i} & \\
 T_{{\mathcal C}/B} \arrow{ru}{\psi} \arrow{rd}[swap]{-\pet} & & R^1 \widetilde{p}_* \C_{\esse} \otimes \OO_{\Ci} \\
 & \widetilde{p}_*(\widetilde{q}^* K_{\Ci/B}(2 {\bf\Delta}_B)) \arrow[hook]{ru}{j} &
 \end{tikzcd}
$$
is commutative. 
\end{prop}

\begin{proof}
To construct the map $j$ in the relative setting, 
notice that we have $$\Omega^1_{{\mathcal S}/B} \,=\, \widetilde{p}^*K_{\Ci/B}
\oplus\widetilde{q}^*K_{\Ci/B}\, ,$$ so, it can be proved as above that there is a
short exact sequence
\begin{equation}\label{jrel}
0 \, \longrightarrow\, \widetilde{p}^{-1}(\OO_{\Ci}) \, \longrightarrow\,
\OO_{\esse}({\bf\Delta}_B) \,\stackrel{\widetilde{T}}{\longrightarrow}\,
\text{image}(\widetilde{T})\, =:\, \widetilde{{\mathcal E}}
\, \longrightarrow\, 0\, , 
\end{equation}
where $\widetilde{T}$ is the composition of the de Rham differential with the
natural projection $$\Omega^1_{{\mathcal S}/B}(2 {\bf\Delta}_B) \,=\,
\widetilde{p}^*K_{\Ci/B} (2 {\bf\Delta}_B) \oplus \widetilde{q}^*K_{\Ci/B}(2
{\bf\Delta}_B) \, \longrightarrow\, \widetilde{q}^*K_{\Ci/B}(2 {\bf\Delta}_B)\, .$$
Applying $\widetilde{p}_*$ to \eqref{jrel} we get that
$$
0\, \longrightarrow\, \widetilde{p}_*\widetilde{\mathcal E}\, \longrightarrow\,
R^1 \widetilde{p}_* (\widetilde{p}^{-1}(\OO_{\Ci})) \, \longrightarrow\, R^1\widetilde{p}_*
(\OO_{\esse}({\bf\Delta}_B)) \, \longrightarrow\, 0\, .
$$
It can be shown likewise that $$\widetilde{p}_* \widetilde{\mathcal E} \,\cong\,
\widetilde{p}_* (\widetilde{q}^* K_{\Ci/B}(2 {\bf\Delta}_B))\, ,$$ and $R^1
\widetilde{p}_* (\widetilde{p}^{-1}(\OO_{\Ci})) \,\cong\, R^1\widetilde{p}_*\C_{\esse}
\otimes \OO_{\Ci}$, so we have
$$j\,:\, \widetilde{p}_* \widetilde{\mathcal E} \,\cong\, \widetilde{p}_* (\widetilde{q}^*
K_{\Ci/B}(2 {\bf{\Delta_B}})) \, \longrightarrow\,
R^1 \widetilde{p}_* (\widetilde{p}^{-1}(\OO_{\Ci}))
\,\cong \,R^1\widetilde{p}_*\C_{\esse} \otimes \OO_{\Ci}\, .$$

The commutativity of the diagram in the proposition now follows from the commutativity of \eqref{cd2}. 
\end{proof}

We summarize the above result as follows. Let
\begin{equation}\label{hs}
0 \,\longrightarrow\, {\mathcal F}^1\,:=\, \pi_*K_{{\mathcal C}/B} \,\longrightarrow\,
R^1\pi_* \C_{\Ci} \,\longrightarrow\, R^1 \pi_* \OO_{\Ci} \,\longrightarrow\, 0
\end{equation}
be the variation of Hodge structure for the family $\pi$ in \eqref{elpi}. Note that
its pullback
\begin{equation}
\label{hodge}
0 \,\longrightarrow\, \pi^*{\mathcal F}^1 \,\longrightarrow\,
\pi^*(R^1\pi_* \C_{\Ci}) \,\longrightarrow\, \pi^*(R^1 \pi_* \OO_{\Ci}) \,\longrightarrow\, 0\, ,
 \end{equation}
to $\mathcal C$ coincides with \eqref{hes}.
We have the following isomorphisms: 
$$R^1\widetilde{p}_* \OO_{\esse}
\,\cong\, \pi^*(R^1 \pi_* \OO_{\Ci})\ \ \text{ and }\ \ R^1 \widetilde{p}_* \C_{\esse}
\otimes \OO_{\Ci} \,\cong\, \pi^*(R^1 \pi_*\C_{\Ci} \otimes \OO_B )\, .$$
Consider the pull-back of \eqref{hodge} via the map $\psi$:
\begin{equation}\label{hodge1}
\begin{tikzcd}
0\arrow{r}& \pi^* {\mathcal F}^1 \arrow{r} & R^1\widetilde{p}_*\C_{\mathcal S} \otimes
\OO_{\mathcal C} \arrow{r}& R^1\widetilde{p}_*\OO_{\mathcal S} \arrow{r} &0 \\
0\arrow{r}& \pi^* {\mathcal F}^1 \arrow{u}{=} \arrow{r}& {\mathcal H} \arrow[hook]{u}
\arrow{r} & T_{{\mathcal C}/B} \arrow{u}{\psi}\arrow{r} \arrow{lu} {i\circ \psi \,=\,
-j \circ \eta}& 0\\
\end{tikzcd}
\end{equation}
where $i$ is the homomorphism in \eqref{fi},
and ${\mathcal H} \,:= \,j (\widetilde{p}_*(\widetilde{q}^* K_{\Ci/B}(2 {\bf\Delta}_B)))$.
\begin{teo}\label{splitting}
The image of the above $C^\infty$ homomorphism $$-j \circ \eta \,=\, i \circ \psi\, :\,
T_{{\mathcal C}/B}\, \longrightarrow\, R^1\widetilde{p}_*\C_{\mathcal S} \otimes
\OO_{\mathcal C}$$ lies in $\mathcal H$, and it gives
a $C^{\infty}$ splitting of the bottom exact sequence 
in the diagram \eqref{hodge1}.
\end{teo}

\begin{proof}
This follows immediately from Proposition \ref{diagram-eta}.
\end{proof}

\section{A canonical projective structure}\label{se3}

Recall Lemma \ref{lem1} and its set-up.
For any $b\, \in\, B$, let $C\, :=\, \pi^{-1}(b)$ be the fiber.
Consider the section $\teta$ constructed in \eqref{weta}. It is anti-invariant
for the involution of $\mathcal L$ (see Proposition \ref{prophi2}),
and $\teta\vert_\Delta$ coincides with section of
${\mathcal L}\vert_\Delta$ in \eqref{evp}, where $\Delta$ as before is
the diagonal divisor in $C\times C$. Therefore, we obtain a $C^\infty$ section
\begin{equation}\label{be}
\beta^\eta\, :\, B\, \longrightarrow\, \widehat{\mathcal V}\, , \ \ b\, \longmapsto\, \teta\vert_{\Delta_3}\, ,
\end{equation}
where $\Delta_3\, \subset\, \pi^{-1}(b)\times \pi^{-1}(b)$ is the second order infinitesimal
neighborhood of the diagonal $\Delta\, \subset\, C\times C$, and
$\widehat{\mathcal V}$ is the fiber bundle constructed in \eqref{whcv}. Now,
$\beta^\eta(b)$ is a projective structure on $C$, by Lemma \ref{lem1}. So $\beta^\eta$ is a $C^\infty$
family of projective structures on the family of curves ${\mathcal C}\,\longrightarrow\, B$.

Now let
\begin{equation}\label{mgp}
\pi\, :\, {\mathcal C}\, \longrightarrow\, \mg
\end{equation}
be the universal family of curves 
over the moduli space of curves; it exists in the orbifold category. Let
\begin{equation}\label{ePi}
\Pi\, :\, \cS\,:=\, {\mathcal C}\times_{\mg} {\mathcal C}\, \longrightarrow\, \mg
\end{equation}
be the projection from the fiber product. Let
\begin{equation}\label{ee}
E \,:=\, \Pi_*{\mathbb L}
\end{equation}
be the direct image, where $\mathbb L$ is defined in \eqref{mbl}. Let
\begin{equation}\label{weta2}
\widetilde{\eta}\, \in\, C^\infty(\mg,\, E)
\end{equation}
be the $C^\infty$ relative version of the section in \eqref{weta}, so $\widetilde\eta$ corresponds to a 
$C^\infty$ homomorphism $\eta$ constructed as in \eqref{releta}.
This section $\widetilde{\eta}$ should be considered in the orbifold category.

The section $\widetilde{\eta}$ in \eqref{weta2} produces a $C^{\infty}$ section of
the fiber bundle $\widehat{\mathcal V}\, \longrightarrow\,
\mg$ in \eqref{whcv}, simply
by restricting a section of ${\mathbb L}\vert_{C\times C}$ over $C\times C$ to the nonreduced
diagonal ${\Delta}_3\, \subset\, C\times C$. The $C^\infty$ sections of
$\widehat{\mathcal V}$ are in a bijective correspondence with the $C^\infty$ families of 
projective structures on $\mathcal C$ (see Lemma \ref{lem1}). Analogously the $C^\infty$ family
of projective structures given by the uniformization of Riemann surfaces
produces a $C^{\infty}$ section (see \eqref{bu})
\begin{equation}\label{bu2}
\beta^u\, :\,
\mg\, \longrightarrow\, \widehat{\mathcal V}
\end{equation}
in the orbifold category.

\begin{teo}\label{lem-ch}
Assume that $g \,\geq \,3$. The $C^{\infty}$ section $\beta^{\eta}: 
\mg\, \longrightarrow\, \widehat{\mathcal V}$ (see
\eqref{be}) of $\widehat{\mathcal V}$ produced by $\widetilde{\eta}$ in \eqref{weta2}
does not coincide with the section $\beta_u$
of $\widehat{\mathcal V}$ in \eqref{bu2}.
\end{teo}

\begin{proof}
As before, let
\begin{equation}\label{cek}
{\mathcal K}\, :=\, {\mathcal K}_{{\mathcal C}/\mg}\, \longrightarrow\, \mathcal C
\end{equation}
denote the relative canonical bundle for the projection $\pi$ in \eqref{mgp}. Let
$$K_{\cS/\mg}\, \longrightarrow\, \cS$$ be the relative canonical line bundle for the
projection $\Pi$ in \eqref{ePi}.
Define the direct images
$$F\,:=\, \Pi_* K_{\cS/\mg} \ \ \text{ and }\ \ 
F_1\,:=\, \Pi_*(K_{\cS/\mg}\otimes {\mathcal O}_{\cS}({\bf\Delta}_{\mg}))$$
over $\mg$. We have $F\,=\,F_1$, because $$H^0(C\times C,\, K_{C\times C})\,=\, H^0(C\times C,\,
K_{C\times C}\otimes {\mathcal O}_{C\times C}(\Delta))$$
for any compact Riemann surface $C$. Hence there is a short exact sequence of sheaves on $\mg$
\begin{equation}\label{sp2}
0\,\longrightarrow\,F\,\longrightarrow\,E\,\longrightarrow\, \cO_{\mg}\,\longrightarrow\, 0\, ,
\end{equation}
where $E$ is constructed in \eqref{ee}. We note that the projection $E\,\longrightarrow\, \cO_{\mg}$ in
\eqref{sp2} sends the smooth section $\widetilde\eta$ of $E$ in \eqref{weta2} to the constant 
function $1$ on $\mg$.

For $\cS$ in \eqref{ePi}, the involution defined by $(x,\, y)\, \longmapsto\, (y,\, x)$ lifts canonically to both
$K_{\cS/\mg}$ and 
$K_{\cS/\mg}\otimes {\mathcal O}_{\cS}({\bf\Delta}_{\mg})$. These lifts of action produce decompositions
$$F\,=\,F^s\oplus F^a\ \ \text{ and }\ \ E\,=\,E^s\oplus E^a$$ into the symmetric and
anti-symmetric parts; so $F^s$ and $E^s$ are the symmetric parts while $F^a$ and $E^a$ are the
anti-symmetric parts. Note that $F^a\,=\, {\rm Sym}^2({\mathcal F}^1)$, where ${\mathcal F}^1\,\longrightarrow
\,\mg$ is the Hodge bundle defined as in \eqref{hs}. From
\eqref{sp2} we have the short exact sequence
\begin{equation}\label{sq}
0\,\longrightarrow\,F^a\,=\, {\rm Sym}^2({\mathcal F}^1)
\,\longrightarrow\,E^a\,\longrightarrow\, \cO_{\mg}\,\longrightarrow\, 0\, .
\end{equation}
The fiberwise multiplication map $H^0(C,\, K_C)\otimes H^0(C,\, K_C)\, \longrightarrow\,
H^0(C,\, K^{\otimes 2}_C)$
produces a ${\mathcal O}_{\mg}$--linear homomorphism
\begin{equation}\label{hh}
\mathbf{m}\, :\, {\rm Sym}^2 ({\mathcal F}^1) \, \longrightarrow\, \pi_*{\mathcal K}^{\otimes 2}
\,=\, \Omega^{1,0}_{\mg}
\end{equation}
(see \eqref{cek}), which is the dual of the differential of the Torelli map $\tau$ in \eqref{itau}.
This map $\mathbf m$ produces a homomorphism
\begin{equation}\label{hh2}
\mathbf{m}'\, :\, \Omega^{0,1}_{\mg}({\rm Sym}^2 ({\mathcal F}^1)) \, \longrightarrow\,
\Omega^{1,1}_{\mg}\, ,
\end{equation}
by tensoring it with $\text{Id}_{\Omega^{0,1}_{\mg}}$. From \eqref{sq} we have the
commutative diagram
\begin{equation}\label{cdl}
\xymatrix{
 0 \ar[r] & {\rm Sym}^2 ({\mathcal F}^1) \ar[r] \ar[d]^{\mathbf m} &E^a \ar[r]\ar[d]^{\mathbf r}
& \cO_{\mg }
\ar[r]\ar[d]^\cong &0\\
0 \ar[r] & T^\ast{\mg} \ar[r] & {\mathcal V}' \ar[r]& \cO_{\mg }\ar[r] &0}
\end{equation}
where ${\mathcal V}'$ is the subbundle of ${\mathcal V}$ (constructed in \eqref{cvn}) generated 
by $\widehat{\mathcal V}$ defined in \eqref{whcv}, the map $\mathbf r$ is the restriction of
sections to the second order infinitesimal neighborhood of
${\bf\Delta}_{\mg}\, \subset\, \cS$ (see \eqref{ePi}) and $\mathbf m$ is the homomorphism in \eqref{hh}.

From Proposition \ref{prophi2} it follows that
$$
\widetilde{\eta}\, \in\, C^\infty(\mg,\, E^a)\, ,
$$
where $\widetilde{\eta}$ is the section in \eqref{weta2}. Note that the composition of maps
$$
\beta^\eta\, :=\, {\mathbf r}\circ \widetilde{\eta}\, \in\, 
C^\infty(\mg,\, {\mathcal V}')\, ,
$$
where $\mathbf r$ is the restriction homomorphism in \eqref{cdl},
is the $C^\infty$ section $\beta^\eta$ in \eqref{be} for the family $\mathcal C$ parametrized
by $B\,=\, \mg$.

Since the surjective homomorphism in \eqref{sq} sends $\widetilde\eta$ to the function $1$ on 
$\mg$, it follows that $\debar\widetilde\eta$ is a section of $\cA^{0,1} ({\rm Sym}^2 ({\mathcal 
F}^1))$. As the homomorphism $\mathbf r$ in \eqref{cdl} is holomorphic,
from the commutativity of \eqref{cdl} we conclude that
\begin{equation}\label{zc}
\mathbf{m}'(\debar\widetilde\eta)\,=\, \overline\partial{\beta}^\eta\, ,
\end{equation}
where ${\mathbf m}'$ is the homomorphism in \eqref{hh2}; since
$\debar\widetilde\eta$ is a section of $\cA^{0,1} ({\rm Sym}^2 ({\mathcal
F}^1))$, it follows that $\mathbf{m}'(\debar\widetilde\eta)$ is a $(1,\,1)$--form on
$\mg$, while the holomorphicity of the projection ${\mathcal V}'\, \longrightarrow\, {\mathcal O}_{\mg}$
in \eqref{cdl} implies that $\overline\partial{\beta}^\eta$ is a $(1,\,1)$--form on $\mg$.

Take a hyperelliptic curve $C\, \in\, \mg$, and take any nonzero $v\, \in\, T_C\mg$ that is
sent to $-v$ by the hyperelliptic involution $\xi$ of $C$; since $g\, \geq\, 3$, such
a $(1,\, 0)$ tangent vector $v$ exists. It can be shown that
$$
(\debar\widetilde\eta (C))(\overline{v})\,=\, 0\, .
$$
Indeed, $\debar\widetilde\eta\, \in\, C^\infty(\mg,\, \Omega^{0,1}_{\mg}
{\rm Sym}^2 ({\mathcal F}^1))$ defines a $C^\infty$ homomorphism
$$
\debar\widetilde\eta (C)\, :\, T^{0,1}\mg \, \longrightarrow\, {\rm Sym}^2 ({\mathcal F}^1)
$$
(the notation is re-used). Since the hyperelliptic involution 
$\xi$ acts trivially on the fiber $({\rm Sym}^2 ({\mathcal F}^1))_C$ and the homomorphism
$\debar\widetilde\eta (C)$ is $\xi$-invariant, it follows that
$$(\debar\widetilde\eta (C))(\overline{v})\,=\, 0\, .$$
Therefore, from \eqref{zc} it follows that $(\overline\partial{\beta}^\eta(C))(\overline{v})\,=\, 0$.
Hence the $(1,1)$--form $\overline\partial{\beta}^\eta$ fails to be positive at $C$. This implies
that $\overline\partial{\beta}^\eta$ is not a nonzero scalar multiple of the
Weil--Petersson $\omega_{wp}$ form on $\mg$, because the form $\omega_{wp}$ is K\"ahler. Consequently,
from \eqref{wp} we conclude that the section $\beta^\eta$ of
$\widehat{\mathcal V}$ produced by $\widetilde{\eta}$ does not coincide with the section
$\beta^u$ in \eqref{bu2} constructed using the uniformization of Riemann surfaces.
\end{proof}

\begin{teo}\label{thm3}
Consider the projective 
structure given by the uniformization of Riemann surfaces. Let $\beta^u\, \in\, 
C^\infty(\mg,\, {\mathcal V}')$ be the corresponding section
(as in \eqref{bu2}). There is no $C^\infty$ section 
$\gamma\, :\, \mg\, \longrightarrow\, E^a$ such that ${\mathbf r}(\gamma)\,=\, \beta^u$, 
where $\mathbf r$ is the restriction map in \eqref{cdl}.
\end{teo}

\begin{proof}
This follows from the proof of Theorem \ref{lem-ch} in a straight-forward way. We omit
the details.
\end{proof}

Let $$\varphi\,:\, {\mathcal U} \,\longrightarrow\, \ag$$ be the universal family
of principally polarized abelian varieties. Let
\begin{equation}\label{abelian}
0 \,\longrightarrow\, {\mathcal F}^1 \,\longrightarrow\, R^1\varphi_* \C_{\mathcal U} \otimes
\OO_{\ag}\,\longrightarrow\,R^1 \pi_* \OO_{{\mathcal U}}\,\cong\,
({\mathcal F}^1)^{\vee} \,\longrightarrow\, 0
\end{equation}
be the exact sequence for the Hodge filtration. The notation ${\mathcal F}^1$ is re-used;
note that ${\mathcal F}^1$ in \eqref{hs} is the pullback of ${\mathcal F}^1$ in
\eqref{abelian} by the map $\tau$ in \eqref{itau}. Let
\begin{equation}\label{ei}
i\,:\, R^1 \pi_* \OO_{{\mathcal U}} \,=\, ({\mathcal F}^1)^\vee \,\longrightarrow\,
R^1\varphi_* \C_{\mathcal U} \otimes \OO_{\ag}
\end{equation}
be the $C^{\infty}$ splitting of \eqref{abelian} given by the Hodge decomposition.

Tensoring \eqref{abelian} with ${\mathcal F}^1$, we get
\begin{equation}\label{echi}
0 \,\longrightarrow\, {\mathcal F}^1 \otimes {\mathcal F}^1 \,\longrightarrow\,
(R^1\varphi_* \C_{\mathcal U}) \otimes {\mathcal F}^1 \,\stackrel{\chi}{\longrightarrow}\,
({\mathcal F}^1)^{\vee}\otimes {\mathcal F}^1 \,\longrightarrow\, 0\, .
\end{equation}
Let $$i'\,=\, i\otimes {\rm Id}_{{\mathcal F}^1} \,:\,
({\mathcal F}^1)^{\vee}\otimes {\mathcal F}^1 \,\longrightarrow\,
R^1\varphi_* \C_{\mathcal U} \otimes {\mathcal F}^1$$ be the $C^{\infty}$ splitting. 
Define the homomorphism
$$
s\, :\, {\OO}_{\ag} \,\longrightarrow\, ({\mathcal F}^1)^{\vee}\otimes{\mathcal F}^1
\,=\, {\rm End}({\mathcal F}^1) \, , \ \ c\, \longmapsto\, c\cdot\text{Id}\, .
$$
Define
$$
{\mathcal G}\,:=\, \chi^{-1}(s({\OO}_{\ag}))\, \subset\, 
(R^1\varphi_* \C_{\mathcal U}) \otimes {\mathcal F}^1\, ,
$$
where $\chi$ is the projection in \eqref{echi}. Now we have the commutative diagram
\begin{equation}\label{ds4}
\begin{tikzcd}
0\arrow{r}& {\mathcal F}^1 \otimes {\mathcal F}^1 \arrow{r} & R^1\varphi_* \C_{\mathcal U} \otimes {\mathcal F}^1 \arrow{r}& ({\mathcal F}^1)^{\vee} 
\otimes {\mathcal F}^1 \arrow{r} &0 \\
0\arrow{r}& {\mathcal F}^1 \otimes{\mathcal F}^1\arrow{u}{=}
\arrow{r}& {\mathcal G} \arrow[hook]{u} \arrow{r} & {\OO}_{\ag} \arrow[hook]{u}{s}\arrow{r}
\arrow{lu} {i'\circ s}& 0\\
\end{tikzcd}
\end{equation}
The image of $i' \circ s$ clearly lies in $\mathcal G$, and the $C^\infty$ homomorphism
$i' \circ s\,:\, {\OO}_{\ag} \,\longrightarrow\,{\mathcal G}$
is a $C^\infty$ splitting of the bottom exact sequence in \eqref{ds4}.
Taking the quotient by $\bigwedge^2 {\mathcal F}^1 $ of the bottom exact sequence in \eqref{ds4}
yields the exact sequence 
\begin{equation}\label{siegel}
0\,\longrightarrow\, {\rm Sym}^2({\mathcal F}^1) \,\longrightarrow\,
{\mathcal G}^+\,:=\, {\mathcal G}/(\bigwedge\nolimits^2 {\mathcal F}^1)
\,\stackrel{f}{\longrightarrow}\, {\OO}_{\ag} \,\longrightarrow\, 0\, ;
\end{equation}
it has a $C^{\infty}$ splitting
$$\sigma\,:=\, q^1\circ i'\circ s \,:\, {\OO}_{\ag} \,\longrightarrow\,{\mathcal G}^+\, ,$$
where $q^1\, :\, {\mathcal G}\, \longrightarrow\, {\mathcal G}^+$ is the projection.

Since the homomorphism $f$ in \eqref{siegel} is holomorphic, for the section $h\,:=\, \sigma(1)$
of ${\mathcal G}^+$,
$$f (\overline{\partial}h) \,=\, \overline{\partial} (f \circ h) \,=\, 
\overline{\partial}(1) \,=\, 0\, ,$$
and consequently from \eqref{siegel} it follows that
\begin{equation}\label{fom}
\omega\,:=\, \overline{\partial}h
\end{equation}
is a $(0,1)$-form on $\ag$ with values in
${\rm Sym}^2 ({\mathcal F}^1) \,=\, T^*\ag$. In other words, $\omega$ is a $(1,1)$--form on $\ag$.

\begin{prop}\label{props}
The $(1,1)$--form $\omega$ in \eqref{fom} is a nonzero constant scalar multiple of
the K\"ahler form $\omega_S$ for the Siegel metric on $\ag$. 
\end{prop}

\begin{proof}
From the construction of $\omega$ it follows that the pullback of $\omega$ to the
Siegel space ${\mathsf H}_g$ is preserved by the action of 
${\rm Sp}(2g, \R)$ on ${\mathsf H}_g$. This implies that $\omega$ is a constant scalar multiple
of the K\"ahler form $\omega_S$ on $\ag$. This scalar factor is nonzero because
the Hodge decomposition $i$ in \eqref{ei} is not holomorphic.
\end{proof}

\begin{teo}\label{thms}\mbox{}
\begin{enumerate}
\item For the $C^\infty$ section $\beta^\eta\, :\, \mg\, \longrightarrow\, \widehat{\mathcal V}$ in \eqref{be},
$$
\overline{\partial} (\beta^\eta)\,=\, \tau^*\omega\, ,
$$
where $\omega$ and $\tau$ are constructed in \eqref{fom} and \eqref{itau} respectively.

\item The $(1,1)$-form $\overline{\partial} (\beta^\eta)$ is a nonzero constant scalar multiple
of $\tau^*\omega_S$, where $\omega_S$ for the Siegel metric on $\ag$.
\end{enumerate}
\end{teo}

\begin{proof}
The second part of the theorem follows immediately from the combination of the first part and
Proposition \ref{props}. We will now prove the first part.

First tensoring the diagram in \eqref{hodge1} with the relative canonical bundle
${\mathcal K}\,:=\, K_{{\mathcal C}/\mg}$, we get
$$
\begin{tikzcd}
0\arrow{r}& \pi^* {\mathcal F}^1\otimes{\mathcal K}
\arrow{r} & R^1\widetilde{p}_*\C_{\mathcal S} \otimes{\mathcal K} \arrow{r}& R^1
\widetilde{p}_*\OO_{\mathcal S} \otimes{\mathcal K}\arrow{r} &0 \\
0\arrow{r}& \pi^* {\mathcal F}^1 \otimes{\mathcal K}\arrow{u}{=} \arrow{r}& {\mathcal H} 
\otimes {\mathcal K}\arrow[hook]{u} \arrow{r} & \OO_{\mathcal C}\arrow{u}{\psi \otimes
{\rm Id}_{\mathcal K}}\arrow{r} \arrow{lu} & 0\\
\end{tikzcd}
$$
together with a $C^{\infty}$ splitting $\OO_{\mathcal C}\,\longrightarrow\,
{\mathcal H} \otimes{\mathcal K}$ given by Corollary \ref{splitting}; here
${\mathcal S}$ and $\widetilde p$ are in \eqref{ePi} and \eqref{dpqt} respectively.
Now applying $\pi_{*}$ yields 
$$
\begin{tikzcd}
0\arrow{r}& {\mathcal F}^1 \otimes {\mathcal F}^1\arrow{r} & R^1\pi_*\C_{\mathcal C} \otimes
{\mathcal F}^1 \arrow{r}& ({\mathcal F}^1)^{\vee} \otimes{\mathcal F}^1 \arrow{r} &0 \\
0\arrow{r}& {\mathcal F}^1 \otimes{\mathcal F}^1 \arrow{u}{=} \arrow{r}&
\pi_*({\mathcal H}\otimes{\mathcal K}) \cong \Pi_* (K_{{\mathcal S}/\mg}(2\Delta_{\mg}))\arrow[hook]{u} \arrow{r}
& \OO_{\mg}\arrow{u}\arrow{r} \arrow{lu} & 0\\
\end{tikzcd}
$$
where the right vertical map is $c\, \longmapsto\, c\cdot\text{Id}$.
Taking quotient of the bottom exact sequence by $\bigwedge^2{\mathcal F}^1$
produces the exact sequence 
\begin{equation}\label{ec}
0 \, \longrightarrow\, {\rm Sym}^2({\mathcal F}^1) \, \longrightarrow\, (\Pi_* (K_{{\mathcal S}/\mg}
(2\Delta_{\mg})))^+ 
\end{equation}
$$
:=\, \Pi_* (K_{{\mathcal S}/\mg}(2\Delta_{\mg}))/(\bigwedge\nolimits^2{\mathcal F}^1)
\, \longrightarrow\, \OO_{\mg}\, \longrightarrow\, 0 
$$
which has a $C^{\infty}$ splitting $ \OO_{\mg}
\, \longrightarrow\, ( \Pi_* (K_{{\mathcal S}/\mg}(2\Delta_{\mg})))^+$ given by $\widetilde{\eta}$
in \eqref{weta2}.
Composing with $$(d\tau)^*\, :\, {\rm Sym}^2 ({\mathcal F}^1) \,\longrightarrow\,
\Omega^{1,0}_{\mg}\, ,$$ where $d\tau$ is the differential of the
Torelli map in \eqref{itau}, we obtain the diagram 
\begin{equation}\label{ec2}
\begin{tikzcd}
0\arrow{r}& {\rm Sym}^2({\mathcal F}^1) \arrow{d} \arrow{r} &
(\Pi_* (K_{{\mathcal S}/\mg}
(2\Delta_{\mg})))^+\arrow{r}\arrow{d}& \OO_{\mg}\arrow{d} {=} \arrow{r} &0 \\
0\arrow{r}& T^*_{\mg} \arrow{r}& {\mathcal V}' \arrow{r} &\OO_{\mg}\arrow{r}& 0\\
\end{tikzcd}
\end{equation}
where ${\mathcal V}'\,\subset\,{\mathcal V}$ is the subbundle in \eqref{cdl} generated by 
$\widehat{\mathcal V}$. We note that diagram in \eqref{ec2} coincides with the one in 
\eqref{cdl}. The bottom exact sequence in \eqref{ec2} admits a $C^{\infty}$ 
splitting $\OO_{\mg} \,\longrightarrow\,{\mathcal V}'$ given by $\beta^\eta$ in \eqref{be}.
The $(1,1)$--form $\overline{\partial} (\beta^\eta)$ over $\mg$ coincides, by construction, 
with $\tau^*\omega$, where $\omega$ and $\tau$ are constructed in \eqref{fom} and 
\eqref{itau} respectively.
\end{proof}

\begin{remark}\label{rgenus2}
Using Theorem \ref{thms}(2) it can be deduced that Theorem
\ref{lem-ch} remains valid for $g\,=\,2$. Indeed, in view of
Theorem \ref{thms}(2), it suffices to prove that
the $(1,1)$--form $\tau^*\omega_S$, where $\omega_S$ is the Siegel $(1,1)$--form on
$\mathsf{A}_2$ and $\tau$ is the map in \eqref{itau}, is not a constant scalar
multiple of the Weil--Petersson K\"ahler form $\omega_{wp}$ on $\mathsf{M}_2$.
To prove that $\tau^*\omega_S$ is not a constant scalar multiple of $\omega_{wp}$,
note that $\tau^*\omega_S$ extends smoothly when a one-parameter family of smooth curves
of genus $2$ degenerates to a reducible stable curve (the limit is two elliptic
curves touching at a point). On the other hand, a theorem of Masur says that
the Weil--Petersson blows up in such a situation (see \cite[p.~624, Theorem 1]{Ma}).
Therefore, $\tau^*\omega_S$ is not a constant scalar multiple of $\omega_{wp}$.
\end{remark}

\section{The class of the intrinsic form}\label{se4}

\subsection{Differentials of the second type on a surface}\label{se4.1}

In this subsection we recall some definitions and results on meromorphic differentials on 
surfaces that will be used in determining the cohomology class of the form 
$\widehat{\eta}$.

Let $S$ be a smooth complex projective surface.
Let $D\,\subset\, S$ be a smooth curve. The vector bundle $\Omega^j_S\otimes {\mathcal O}_S(nD)$
will be denoted by $\Omega^j(nD)$.

Let $\cA^{p,q}$ be the sheaf of smooth differential forms of type $(p,\,q)$ on $S.$
Let $$\cA^m\,=\,\bigoplus _{p+q=m} \cA^{p,q}$$ be the sheaf of the $m$--forms,
with $\cC^\infty_S\,=\,\cA^0\,=\,\cA^{0,0}.$
Let $\cA^{p,q}(nD)$ be the sheaf having a pole of order at most $n$ on $D$; more precisely,
if $x\,=\,0$ is a local equation of $D$ on $U \bigcap D$ then
$$\omega\,\in\, \cA^{p,q}(nD)(U)\,\iff\, x^n\omega\,\in\, \cA^{p,q}(U).$$
Now define $\cA^m(D)\,=\,\bigoplus _{p+q=m} \cA^{p,q}(D).$
We consider the complexes
\begin{equation}\label{smooth}
\cC^\infty\,\stackrel{d}{\longrightarrow}\, \cA^{1,0}(D)\oplus \cA^{0,1}
\,\stackrel{d}{\longrightarrow}\, \cA^{2}(2D)
\end{equation}
\begin{equation}
\cO_S \,\stackrel{d}{\longrightarrow}\, \Omega^1(\log D)
\,\stackrel{d}{\longrightarrow}\,\Omega^2(D) \label{hol}
\end{equation}
and
\begin{equation}
\cO_S\,\stackrel{d}{\longrightarrow}\, \Omega^1(D)\,\stackrel{d}{\longrightarrow}\,
\Omega^2(2D)\, .\label{hol2}
\end{equation}

\begin{lemma}
The cohomology sheaves of all the above complexes are isomorphic to $$\bC_D\,=\,j_\ast \bC\, ,$$
where $j\,:\,D\,\stackrel{\longrightarrow}\,S$ is the inclusion map.
\end{lemma}

\begin{proof}
We investigate the smooth case in \eqref{smooth}. Define
$$\cN\,:=\, \ker 
(d: \cA^{1,0}(D) + {\cA^{0,1}} \to \cA^2(2D))\, .$$ 
Let $U\, \subset\, S$ be an open subset with coordinate function $(x,\, y)$ such that 
$x\,=\,0$ is the equation of $D\bigcap U$. Take $\omega\,\in\, \cN(U)$.
We have: $$\omega \,=\, \frac {f(x,y)dx+g(x,y)dy}{x}+ \phi\, ,$$
where $\phi$ is a smooth $1-$form. The terms with poles of $\debar\omega$ in the coefficients 
of $dx\wedge d\overline x$ and $dx\wedge d\overline y$ are $-\frac{\partial 
f}{\partial\overline x} $ and $-\frac{\partial f}{\partial\overline y}$ respectively. Hence, 
if $d\omega\,=\,0$, then $f(x,y)$ is holomorphic. Using the Taylor expansion of $f$ with 
respect to $x$ we can write $$\omega \,=\, \frac {h(y)dx+g(x,y)dy}{x}+ \widetilde{\phi}\, .$$ 
The polar terms of $d\omega$ in $dx\wedge dy$ is
$$
\left(-\frac{h'(y)}x 
+\frac1x\frac {\partial g}{\partial x} -\frac g{x^2}\right)dx\wedge dy\, ,
$$
and hence if $d\omega\,=\,0$, then
$g(x,y)\,=\,0$ and $h'(y)\,=\,0$, in which case $\omega\,=\,\frac{\lambda dx}{x}+\widetilde{\phi}$
with $\widetilde\phi$ being closed, and hence $\omega$ is locally exact. The residue map
$$\rm{res}\,:\,\cN\,\longrightarrow\, {\bC}_D\, ,\ \ 
\omega\,\longmapsto\,\, \lambda$$ is well defined, and its kernel is $d\cC^\infty.$
\end{proof}

Consider \eqref{hol} and \eqref{hol2}. Define
$$ \cN_h\,=\,\ker (d\,:\, \Omega^1(D)\,\longrightarrow\, \Omega^2(2D))\, \subset\, \Omega^1(D)\, ;$$
note that it is not a coherent sheaf. Now the exact sequence in \eqref{hol2} yields
\begin{equation}\label{ec1}
0\,\longrightarrow\, \cN_h\,\longrightarrow\, \Omega^1(D)\,\longrightarrow\,
\Omega^2(2D)\,\longrightarrow\, 0
\end{equation}
$$ 0\,\longrightarrow\, \bC_S\,\longrightarrow\, \cO_S\stackrel{d}{\longrightarrow}\,
\cN_h
\,\longrightarrow\, \bC_D\,\longrightarrow\, 0\, .$$
We also get
\begin{equation}\label{2ec}
0\,\longrightarrow\, L_h\,\longrightarrow\, \cN _h\,\longrightarrow\, \bC_D\,\longrightarrow\, 0
\end{equation}
\begin{equation}\label{3ec}
0\,\longrightarrow\, \bC_S\,\longrightarrow\, \cO_S\,\longrightarrow\, L_h
\,\longrightarrow\, 0\, ,
\end{equation}
where $L_h \,=\, d \cO_S$. 
Let us construct a homomorphism $$\Gamma\,:\, H^0(S,\,\Omega^2(2D)) \,\longrightarrow\, H^1(D,\,\bC)$$
as follows: 
Take $\zeta\,\in\, H^0(S,\, \Omega^2(2D))$. Set $\Gamma(\zeta)\,\in\, H^1(D,\,\bC)$ to be the image
of its coboundary $\partial \zeta \,\in\, H^1(S,\, \cN_h)$ (see \eqref{ec1}) under the homomorphism
$H^1(S,\, \cN_h)\,\longrightarrow\, H^1(\bC_D) \, \cong\, H^1(D,\, \bC)$ (see \eqref{2ec}).
 
\begin{defin}\label{defst}
A form $\zeta\,\in\, H^0(S,\,\Omega^2(2D))$ is of \textit{second type} if
$$
[\zeta]\,=\, j^\ast [\gamma]\,\in\, H^2(S 
\setminus D,\,\bC)
$$
with $[\gamma]\,\in\, H^2(S,\,\bC)$, where
$j\,:\, S \setminus D \,\longrightarrow\, S$ is the 
inclusion map. Equivalently, $\zeta$ is a differential of the second type if and
only if $\Gamma(\zeta)\,=\, 0.$ 
\end{defin}

For any $\zeta$ of second type, since $\Gamma (\zeta)\,=\,0,$ from \eqref{2ec} it follows 
that $\partial\zeta\,=\,r(\beta)$, where $\beta\,\in\, H^1(S,\, L_h)$ and $r\,:\,
H^1(S,\, L_h)\, \longrightarrow\, H^1(S,\, \cN _h)$ is the homomorphism of cohomologies given
by the injective homomorphism of sheaves in \eqref{2ec}. The coboundary homomorphism
$\partial\,:\, H^1(S,\,L_h)\,\longrightarrow\, H^2(S,\,\bC_S)$ (see \eqref{3ec}) gives a class
$\partial (\beta)\,=\,[\gamma]\,\in\, H^2(S,\bC_S).$ By construction, we have
$j^\ast([\gamma])\,=\,[\zeta]$, where $j$ is the inclusion map
in Definition \ref{defst}, and $$[\gamma]\,\longmapsto\, 0 \,\in \,H^2(S,\, \cO_S)\, .$$
This means that the $(0,\,2)$ part of $[\gamma]$ vanishes.
Consequently, using the Hodge decomposition for $S$,
\begin{equation}\label{gai}
[\gamma]\,=\,\gamma^{2,0}+\gamma^{1,1}\, ,
\end{equation}
where $\gamma^{2,0}$ is holomorphic and
$\gamma^{1,1}$ is harmonic of type $(1,\,1)$. Therefore
$$\zeta'\,:=\, \zeta-\gamma^{2,0}$$
is a holomorphic differential of second type of pure type $(1,\, 1)$.
 
\begin{prop}\label {doppio}
Let $\zeta'\,\in\, H^0(S,\, \Omega^2_S(2D))$ be a holomorphic differential of the second type and
of pure type $(1,\,1)$. Then there is a class $\alpha\,\in\, H^0(S,\,\cA^{1,0}(D))$ such that
$$\zeta'-\gamma^{1,1}\,=\, d\alpha\, ,$$ that is $\partial\alpha\,=\,\zeta'$
and $\debar \alpha\,=\,-\gamma^{1,1},$ where $\gamma^{1,1}$ is a harmonic $(1,1)$ form on $S.$
\end{prop}

\begin{proof} 
From the sequence \eqref{smooth}
$$\cC^\infty\,\stackrel{d}{\longrightarrow}\,
 \cA^{1,0}(D)\oplus \cA^{0,1} \,\stackrel{d}{\longrightarrow}\, \cA^2(2D)$$
we get an exact sequence
\begin{equation} 0\,\longrightarrow\, \bC_S\,\longrightarrow\, \cC^\infty
\,\stackrel{d}{\longrightarrow}\, \cN \,\longrightarrow\, \bC_D\,\longrightarrow\, 0\, .
\label{smooth1}
\end{equation}
Let $\gamma^{1,1}$ be a $(1,\,1)$ harmonic form on $S$ such that $j^*([\gamma])\,=\,[\zeta'].$
The form $$\zeta'-\gamma^{1,1}\,\in\, H^0(S,\,d(A^{1,0}(D)\oplus \cA^{0,1}))$$
maps to zero in $ H^1(D,\, \bC)$. We find then
$$\beta\,=\,\beta^{1,0}+\beta^{0,1}\,\in\, H^0(S,\,\cA^{1,0}(D)\oplus \cA^{0,1})$$
such that $d\beta\,=\,\zeta'-\gamma^{1,1}$.
Since $\zeta'$ is of type $(2,\,0)$,
$$
\zeta'\,=\, \partial \beta^{1,0}\ \ \ \text{ and } \, \ \ -\gamma^{1,1}\,=\, \debar \beta^{1,0}+
\partial \beta^{0,1}\, , \ \ \debar \beta^{0,1}\,=\,0\, .
$$
Now $\partial \beta^{0,1}$ is a smooth form that is $\partial$ exact and $\debar$ closed. By
$\partial\debar$ lemma, there is a smooth function $f$ such that
$\partial \beta^{0,1}\,=\,\partial\debar f$. Set $$\alpha\,=\, \beta^{1,0}-\partial f\,\in\,
H^0(S,\, A^{1,0}(D))\, .$$ We get that
$$\partial\alpha\,=\, \zeta'\ , \ \ \debar \alpha\,=\, -\gamma^{1,1}\, ,$$
and the proof is complete.
\end{proof}

\subsection{The class of $\widehat{\eta}$}\label{se4.2}

Let $C$ be a smooth complex projective curve of genus $g$. As before, $p$ and $q$ are the 
projections of $S\,:=\, C\times C$ to the first and second factors respectively. Also, as 
before $\Delta\, \subset\, S$ is the reduced diagonal divisor. Recall that the line bundle $\mathcal L$
in \eqref{dcl} is identified with $\Omega^2_S(2\Delta)\,:=\, \Omega^2_S\otimes {\mathcal O}_S(2\Delta)$.

{}From the exact sequence of homology groups for the pair $(S,\, \Delta)$ and Poincar\'e and Lefschetz duality, it 
follows that the homomorphism $j^\ast\,:\, H^2(S,\,\bZ)\,\longrightarrow\, H^2(S\setminus \Delta,\,\bZ)$ is 
surjective with the kernel of $j^\ast$ being generated by the class of the diagonal. Consequently, every element of 
$H^0(\Omega^2_S(2\Delta))$ is of second type. In particular, the form $\widehat{\eta}$ constructed in 
\eqref{weta} is of second type.

\begin{teo}\label{1,1}
All the elements of $H^0(S,\, \Omega^2_S(2\Delta))$ with cohomology class in $H^2(S\setminus \Delta,\,\bC)$ of pure type $(1,\,1)$
are actually contained in the line ${\mathbb C}\cdot \widehat{\eta}$.
\end{teo}

\begin{proof}
Write $[\widehat{\eta}] \,=\, j^*([\gamma])$ and $\gamma \,=\, \gamma^{2,0} + \gamma^{1,1}$
as in \eqref{gai} and Definition \ref{defst}. By Proposition \ref{doppio} 
there is a form $\alpha\,=\,\alpha^{1,0}\,\in\, C^\infty(\cA^{1,0}(\Delta))$ such that
$$\widehat{\eta}-\gamma^{2,0}\,=\,\partial \alpha\ \ \ \text{ and } \ \ \ \gamma^{1,1}
\,=\, -\debar \alpha\, .$$
Since $\debar \alpha$ is smooth it follows that the polar part of $\alpha$ is smooth, that is,
in local coordinates around the diagonal,
$$\alpha\,=\, \frac{d(w+z)}{z-w}+ h_1(z,w)dz+h_2(z,w)dw\, ,$$
where $h_1$ and $h_2$ are smooth functions. We want to prove that $\gamma^{2,0} \,=\,0$. 

Since the decomposable forms generate the $(2,\,0)$ cohomology, it suffices to prove that
\begin{equation}\label{tp}
\int_S \gamma \wedge\overline{\omega}_1\wedge \overline{\beta}_2\,=\,0\, ,
\end{equation}
where $ \overline{\omega}_1 \,=\, p^\ast\overline{\omega}$ and $\overline{\beta}_2
\,=\, q^\ast\overline{\beta}$ with $\omega$ and $\beta$ being holomorphic $1$--forms on $C$. 

We write $$\widehat{\eta}\,=\,\gamma+d\alpha.$$ 

Let $U_r$ be a tubular neighborhood of the diagonal $\Delta$ and $\chi_r$
the characteristic function of the complement $S\setminus U_r$.
We will show that 
\begin{equation}\label{tp1}
\lim_{r\to 0}\int_S \chi_r \widehat{\eta} \wedge \overline{\omega}_1\wedge \overline{\beta}_2\,=\,0
\end{equation}
and 
\begin{equation}\label{tp2}
\lim_{r\to 0}\int_S \chi_r d\alpha \wedge \overline{\omega}_1\wedge \overline{\beta}_2\,=\,0\, .
\end{equation}
Note that \eqref{tp1} and \eqref{tp2} together imply \eqref{tp}.

Take $\{W,\, w\}$, where $W\, \subset\, C$ is an open subset and $w$ is a holomorphic coordinate function
on $W$; define $$V\,:=\,q^{-1}(W) \,=\, C \times W\, .$$ On $V$, we have
$\widehat{\eta}\,=\, \mu \wedge dw$ where $\mu$ is a $(1,\,0)$ meromorphic form with pole on
the diagonal of $W\times W$. Using Fubini's theorem,
$$\int_V \chi_r \widehat{\eta} \wedge \overline{\omega}_1\wedge \overline{\beta}_2
\,=\, \int_W( \int_{C \times\{t\}} \chi_r\mu_t\wedge\overline{\omega})dw\wedge \overline{\beta}\, .$$
The form $ \chi_r \mu_t\wedge\overline{\omega}$ is can be shown to be exact.
Indeed, this follows from the fact that $\mu_t$ is the form defined in \eqref{phi},
whose local expression in a coordinate $z$ centered in $t$ is 
$\mu_t \,=\,( \frac{1}{z^2}+h(z)) dz \,=\, \partial(-\frac1z+f(z))$ with $f(z)$ holomorphic. So,
$$\int_{C \times\{t\}} \chi_r\mu_t\wedge\overline{\omega}
\,=\, \int_{\partial \Delta_r}(-\frac1z+f(z))\overline{\omega}\, .$$
Now writing $\omega\,=\, g( \overline{z})d\overline z$,
\begin{equation}\label{intl}
\int_{\partial \Delta_r}\left(-\frac1z+f(z)\right)g(\overline{z}) d\overline z
\end{equation}
$$
= \, \frac{1}{\sqrt{-1}}
\int_0^{2\pi}\left(\frac{-1}{re^{\sqrt{-1}\theta}}re^{-\sqrt{-1}\theta}g(re^{-\sqrt{-1}\theta})+
f(re^{\sqrt{-1}\theta})g(re^{-\sqrt{-1}\theta})re^{-\sqrt{-1}\theta}d\theta\right)\, .
$$
The integrand functions $$ -g(re^{-\sqrt{-1}
\theta})e^{-2\sqrt{-1}\theta}+f(r\cdot e^{\sqrt{-1}\theta})g(r\cdot e^{-\sqrt{-1}\theta})re^{-\sqrt{-1}
\theta}$$
are integrable and bounded. The limit of the integral in \eqref{intl} for $r\,\to\, 0$ is $0$
since $g$ is anti-holomorphic and $g(r\cdot e^{-\sqrt{-1}\theta})=g(0)+r(\widetilde{g}(r
\cdot e^{-\sqrt{-1}\theta})).$
Then we can pass the limit under the integral sign:
$$\lim_{r\to 0}\int_V \chi_r \widehat{\eta} \wedge \overline{\omega}_1\wedge \overline{\beta}_2
\,=\,\lim_{r\to 0}\left(\int_W\int_{C \times\{t\}} \chi_r\mu_t\wedge\overline{\omega}\right)dw\wedge
\overline{\beta}$$
$$=\, \int_W\left(\lim_{r\to 0} \int_{C \times\{t\}} \chi_r\mu_t\wedge\overline{\omega}\right)dw\wedge
\overline{\beta}\,=\,0\, .$$
This proves \eqref{tp1}.

To prove \eqref{tp2},
$$\lim_{r\to 0}\int_S \chi_r d\alpha \wedge \overline{\omega}_1\wedge \overline{\beta}_2
\,=\,\lim_{r\to 0} \int_{\partial U_r} \alpha \wedge \overline{\omega}_1\wedge \overline{\beta}_2. $$
Cover the diagonal with a finite number of products of disks
$A_i\times B_i$, biholomorphic to $\{(z,\,w)\,\in\, {\bC}^2 \,\mid\, |z|\,<\,1,\,\, |w|\,<\,1\}$ (with compact closure).
We take $r$ small enough in a way that 
$\partial U_r \,\subset\,\bigcup_i A_i\times B_i$ and we may assume that 
$\partial U_r \bigcap (A_i\times B_i)$ corresponds to ${\mathcal B}_r\,=\,\{(z,\,w)\,\in\, {\bC}^2\,\mid\, |z-w|\,=\,r\},$
hence $w\,=\,z+re^{\sqrt{-1}\theta}.$ We may also assume that $\alpha $ is of the type
$$\alpha\,=\, \frac{d(w+z)}{z-w}+ h_1(z,w)dz+h_2(z,w)dw$$
in local coordinates, where $h_1$ and $h_2$ are smooth functions. Write moreover
$\overline\omega_1\,=\,f_1(\overline{z})d\overline z$ and $\overline\alpha_2\,=\,
f_2(\overline{w})d\overline w$ with the $f_i$ being anti-holomorphic.

We have
\begin{equation}\label{es1}
\int_{{\mathcal B}_r} \frac{d(w+z)}{z-w}f_1(\overline{z})d\overline zf_2(\overline{w})d\overline{ w}
\end{equation}

$$
=\, 2\sqrt{-1} \int_{{\mathcal B}_r} f_1(\overline{z}) f_2(\overline{w}) e^{-2\sqrt{-1} \theta} dz d\overline{z}
d\theta
$$
$$
=\, 2\sqrt{-1} \bigg( \int_{|z|<1} \int_0^{2\pi}f_1(\overline{z}) f_2(\overline{z}) e^{-2\sqrt{-1}\theta} dz
d\overline{z} d \theta +
 \int_{|z|<1} \int_0^{2\pi}
rf_1(\overline{z}) h(\theta, \overline{z}) e^{-2\sqrt{-1}\theta} dz d\overline{z} d\theta\, \bigg) . 
$$
Now we have $\int_{|z|<1} f_1(\overline{z}) f_2(\overline{z}) dz d \overline{z}
\int_0^{2 \pi} e^{-2\sqrt{-1} \theta} d\theta \,=\, 0$ and
$$\int_{|z|<1} \int_0^{2\pi}\Big\vert rf_1(\overline{z}) h(\theta, \overline{z}) e^{-2\sqrt{-1}\theta}
dz d\overline{z} d\theta\Big\vert \,=\, r \int_{|z|<1}\left(\int_0^{2\pi}\Big\vert f_1(\overline{z})
h (\theta, \overline{z})\Big\vert d\theta\right) dz d\overline{z} \,=\, rc\, ,$$
where $c$ is a constant. Then the limit of the integral in \eqref{es1} is zero as $r\,\to\, 0$. 

It now follows that we have to evaluate only 
$$\int_{{\mathcal B}_r} (h_1(z,w)dz+h_2(z,w)dw) f_1(\overline{z})f_2(\overline{w}) d\overline z d\overline w\, .$$

We compute the two terms separately, let $G\,=\,h_1(z,w)f_1(\overline{z})f_2(\overline{w})$: 
$$\int_{{\mathcal B}_r} \Big\vert G(z,w)dzd\overline z d\overline w\Big\vert\,=\,
\int_{{\mathcal B}_r}\Big\vert (G(z,\theta)dzd\overline 
z)(-\sqrt{-1}re^{-\sqrt{-1}\theta}d\theta)\Big\vert
$$
$$
=\, r \int _{|z| <1}(\int_0^{2\pi}\Big\vert G(z,\theta)\Big\vert d\theta) dzd\overline z\,=\,rk\, ,
$$
where $k$ is a constant, and similarly for the second term. Then the limit as $r\,\to\, 0$ is zero 
and therefore summing $$ \lim_{r\to 0}\Big\vert\int_S \chi_r d\alpha \wedge \overline{\omega}_1\wedge 
\overline{\beta}_2\Big\vert\,\,\leq\,\, \lim_{r\to 0}\sum_i\int_{A_i\times B_i}\Big\vert\chi_r d\alpha \wedge 
\overline{\omega}_1\wedge \overline{\beta}_2\Big\vert\,=\,0\, .$$
This proves \eqref{tp2}.
\end{proof}

\section*{Acknowledgements}

We thank the referee for helpful comments to improve the exposition.
We warmly thank Alessandro Ghigi for several useful discussions on the properties of the form 
$\widehat{\eta}$. We are grateful to Scott Wolpert for the argument in Remark \ref{rgenus2}. 
We thank Richard Hain for a helpful correspondence. I. Biswas is supported by a J. C. Bose 
Fellowship, and school of mathematics, TIFR, is supported by 12-R$\&$D-TFR-5.01-0500. E. 
Colombo, P. Frediani and G. P. Pirola are partially supported by by MIUR PRIN 2017 ``Moduli 
spaces and Lie Theory'' and by INdAM (GNSAGA). P. Frediani and G. P. Pirola are partially 
supported by MIUR, Programma Dipartimenti di Eccellenza (2018-2022) - Dipartimento di 
Matematica ``F. Casorati'', Universit\`a degli Studi di Pavia.


\begin{thebibliography}{ZZZZZ}

\bibitem[BR1]{BR1} I. Biswas and A. K. Raina, Projective structures on a Riemann surface, {\it 
Inter. Math. Res. Not.} (1996), no. 15, 753--768.

\bibitem[BR2]{BR2} I. Biswas and A. K. Raina, Projective structures on a Riemann surface, II, {\it 
Inter. Math. Res. Not.} (1999), no. 13, 685--716.

\bibitem[CF]{cf} E.~Colombo and P.~Frediani, Siegel metric and
curvature of the moduli space of curves, {\em Trans. Amer. Math. Soc.}
{\bf 362} (2010), 1231--1246.

\bibitem[CFG]{cfg} E.~Colombo, P.~Frediani, and A.~Ghigi, On totally geodesic 
submanifolds in the {J}acobian locus, {\em International Journal of Mathematics}, {\bf 26}
(2015), no. 1, 1550005 (21 pages).

\bibitem[CPT]{cpt} E.~Colombo, G.~P. Pirola, and A.~Tortora, Hodge-{G}aussian maps, {\em 
Ann. Scuola Norm. Sup. Pisa Cl. Sci.} {\bf 30} (2001), 125--146.

\bibitem[De]{Demailly} J.-P. Demailly, {\it Complex analytic and differential geometry},\\
https://www-fourier.ujf-grenoble.fr/$\sim$demailly/documents.html.

\bibitem[GPT]{gpt} A.~Ghigi, G.~P. Pirola and S. Torelli, Totally geodesic subvarieties in the moduli 
space of curves. {\em arXiv:1902.06098}. To appear in Communications in Contemporary Mathematics. https://doi.org/10.1142/S0219199720500200. 

\bibitem[FP]{fp} P. Frediani and G.~P. Pirola, On the geometry of the second fundamental form of the 
Torelli map. {\em arXiv:1907.11407}. To appear in Proceedings of the AMS. https://doi.org/10.1090/proc/15291. 

\bibitem[Gu1]{Gu} R. C. Gunning, {\it On uniformization of complex manifolds: the
role of connections}, Princeton Univ. Press, 1978.

\bibitem[Gu2]{gun} R.C.~{Gunning}, {\it Some topics in the function theory of compact Riemann 
surfaces},\\
{https://web.math.princeton.edu/$\sim$gunning/book.pdf}.

\bibitem[Ha]{Ha} R. Hartshorne, \textit{Algebraic geometry}, Graduate Texts in Mathematics, No. 52.
Springer-Verlag, New York-Heidelberg, 1977.

\bibitem[Iv]{Iv} N. V. Ivanov, Projective structures, flat bundles and K\"ahler metrics on moduli 
spaces, {\it Math. USSR-Sb.} {\bf 61} (1988), 211--224.

\bibitem[Ko]{kobayashi-vector}
S.~Kobayashi, {\em Differential geometry of complex vector bundles}, volume~15 of
{\em Publications of the Mathematical Society of Japan}.
Princeton University Press, Princeton, NJ, 1987.
{Kan\^o} Memorial Lectures, 5.

\bibitem[Ma]{Ma} H. Masur, The extension of the Weil-Petersson metric to the boundary of 
Teichm\"uller space, {\it Duke Math. Jour.} {\bf 43} (1976), 623--635.

\bibitem[Mc]{Mc} C. T. McMullen, The moduli space of Riemann surfaces is K\"ahler hyperbolic,
{\it Ann. of Math.} {\bf 151} (2000), 327--357.

\bibitem[OS]{os} F.~Oort and J.~Steenbrink, The local {T}orelli
problem for algebraic curves. In {\em Journ\'ees de
{G}\'eometrie {A}lg\'ebrique d'{A}ngers, {J}uillet
1979/{A}lgebraic {G}eometry, {A}ngers, 1979}, pages 157--204.
Sijthoff \& Noordhoff, Alphen aan den Rijn, 1980.

\bibitem[ZT]{ZT} P. G. Zograf and L. A. Takhtadzhyan, On the uniformization of Riemann surfaces and on 
the Weil-Petersson metric on the Teichm\"uller and Schottky spaces,
{\it Math. USSR-Sb.} {\bf 60} (1988), 297--313.

\end{thebibliography}
\end{document}